\documentclass[12pt, reqno]{amsart}
\usepackage[T1]{fontenc}
\usepackage[margin=1 in]{geometry}
\usepackage{amsfonts,mathrsfs,amssymb,amsthm,mathtools, extarrows,textcmds}
\usepackage{array}
\usepackage{lmodern}

\usepackage{tikz-cd} 
\usetikzlibrary{decorations.pathmorphing}
\usepackage{xcolor, xspace} 

\usepackage{setspace}
\DeclareFontFamily{U}{futm}{}
\DeclareFontShape{U}{futm}{m}{n}{
  <-> s * [1.0] fourier-bb
  }{}
\DeclareSymbolFont{Ufutm}{U}{futm}{m}{n}
\DeclareSymbolFontAlphabet{\mathbb}{Ufutm}

\newcommand{\defi}[1]{\textit{\textsf{#1}}}
\usepackage[scaled=0.86]{helvet}

\usepackage[shortlabels]{enumitem}
\setlist[enumerate]{leftmargin=0pt,itemindent=1cm}

\tikzset{%
  symbol/.style={
    draw=none,
    every to/.append style={
      edge node={node [sloped, allow upside down, auto=false]{$#1$}}
    },
  },
}

\usepackage{float}
\restylefloat{table}

\usepackage{aliascnt} 
\usepackage[colorlinks=true,linkcolor={red!70!black},citecolor={green!50!black},urlcolor={blue!70!black},citebordercolor={0 0 1},urlbordercolor={0 0 1},linkbordercolor={0 0 1},pagebackref]{hyperref}
\usepackage[normalem]{ulem} 
\usepackage{cleveref}
\usepackage[alphabetic,msc-links]{amsrefs}
\usepackage{bm}

      \newtheorem{thm}{Theorem}[section]
      \newtheorem{lem}[thm]{Lemma}
      \newtheorem{prop}[thm]{Proposition}

      \newtheorem*{lem*}{Lemma}
\newtheorem*{thm*}{Theorem}
\newtheorem*{prop*}{Proposition}
      \theoremstyle{definition}
      \newtheorem{emp}[thm]{}
      
      \newtheorem{exmp}[thm]{Example}
      \newtheorem{rem}[thm]{Remark}
      \newtheorem{quest}[thm]{Question}

\theoremstyle{definition}
{
\newtheorem*{exmp*}{Example}
\newtheorem*{defn*}{Definition}
\newtheorem*{rem*}{Remark}
\newtheorem*{ans*}{Answer}
\newtheorem*{quest*}{Question}
}

\numberwithin{equation}{section}
\numberwithin{figure}{section}
\numberwithin{table}{section}
\makeatletter
    \let\c@equation\c@thm
    \let\c@figure\c@thm
    \let\c@table\c@thm
\makeatother

\tikzset{
    labl/.style={anchor=south, rotate=90, inner sep=.5mm}
}

\usepackage{adjustbox}
\newcommand{\mbf}[1]{\mathbf{#1}}
\newcommand{\inv}{^{-1}}

\newcommand{\Q}{\mathbf{Q}}

\newcommand{\Z}{\mathbf{Z}}
\newcommand{\C}{\mathbf{C}}

\newcommand{\p}{\mathbf{P}}
\newcommand{\X}{\mathbb{X}}

\newcommand{\mm}{\mathbb{M}}
\newcommand{\D}{\mathbb{D}}
\newcommand{\ee}{\mathbb{E}}

\let\L\relax
\newcommand{\L}{\mathbb{L}}

\newcommand{\J}{\mathbb{J}}
\newcommand{\ff}{\mathbb{F}}

\newcommand{\kk}{\mathbb{K}}
\newcommand{\calm}{\mathcal{M}}

\newcommand{\calp}{\mathcal{P}}

\newcommand{\calx}{\mathcal{X}}

\newcommand{\calb}{\mathcal{B}}

\newcommand{\calj}{\mathcal{J}}

\newcommand{\A}{\mathbb{A}}

\newcommand{\prsup}[2]{\prescript{#1}{}{#2}}
\newcommand{\prsub}[2]{\prescript{}{#1}{#2}}

\newcommand{\scrc}{\mathscr{C}}
\newcommand{\mf}{\mathfrak}

\newcommand{\inj}{\lhook\joinrel\longrightarrow}
\newcommand{\surj}{\twoheadrightarrow}

\newcommand{\gal}{\textup{Gal}}

\newcommand{\Hom}{\textup{Hom}}

\newcommand{\id}{\textup{Id}}

\newcommand{\iso}{\xrightarrow{\raisebox{-0.7ex}[0ex][0ex]{$\;\sim\;$}}}

\newcommand{\disc}{\textup{disc}}
\newcommand{\qbar}{\overline{\Q}}

\newcommand*{\DashedArrow}[1][]{\mathbin{\tikz [baseline=-0.25ex,-latex, dashed,#1] \draw [#1] (0pt,0.5ex) -- (1.3em,0.5ex);}}

\newcommand{\wt}[1]{\widetilde{#1}}

\newcommand{\ul}[1]{\underline{#1}}
\newcommand{\ol}[1]{\overline{#1}}

\newcommand{\kbar}{\ol{K}}

\newcommand{\ceq}{\coloneqq}
\newcommand{\res}[1]{\textup{Res}_{#1}}
\newcommand{\fetk}{finite \'etale $K$-algebra }

\newcommand{\too}{\longrightarrow}
\newcommand{\mtoo}{\longmapsto}

\newcommand{\triv}{\mathbf{1}}
\newcommand{\ind}[2]{\textup{Ind}^{#1}_{#2}}

\newcommand{\Et}{\textup{\textsf{\'Et}}}

\let\gg\relax
\newcommand{\gg}[1]{G_{#1}}

\newcommand*\Bell{\ensuremath{\boldsymbol\ell}}

\newcommand{\divzero}[1]{\textup{div}_0(#1)}

\newcommand{\V}[2]{V({#1}|{#2})}
\newcommand{\W}[1]{W(#1)}

\newcommand{\sch}[2]{#1\textup{--\textsf{sch}}(#2)}

\newcommand{\ksch}{K\textup{--\textsf{sch}}}

\newcommand{\qgk}{\Q G_K}
\newcommand{\qg}{\Q G}
\newcommand{\qgm}{\Q G_{\mm}}
\newcommand{\glk}{G_{L/K}}

\renewcommand{\arraystretch}{1.5}

\usepackage{float}
\restylefloat{table}

\tikzset{
labl1/.style={anchor=north, rotate=90, inner sep=1.2mm}
}

\DeclareMathOperator{\chr}{\textup{char}}
\DeclareMathOperator{\spec}{\textup{Spec}} 

\DeclareMathOperator{\Div}{\textup{Div}}

\DeclareMathOperator{\rk}{\textup{rk}}

\DeclareMathOperator{\aut}{\textup{Aut}}

\DeclareMathOperator{\codim}{\textup{codim}}
\DeclareMathOperator{\sym}{\textup{Sym}}
\DeclareMathOperator{\End}{\textup{End}}

\DeclareMathOperator{\stab}{\textup{Stab}}
\DeclareMathOperator{\trdeg}{\textup{trdeg}}

\DeclareMathOperator{\isom}{\textup{Isom}}

\DeclareMathOperator{\twist}{\textup{Twist}}

\DeclareMathOperator{\mor}{\textup{Mor}}

\let\deg\relax
\DeclareMathOperator{\deg}{\textup{deg}}

\let\dim\relax
\DeclareMathOperator{\dim}{\textup{dim}}

\usepackage{microtype}

\subjclass{11F80, 11G10, 11G15, 11G30, 11G40, 12G05, 14G10, 14K22}

\begin{document}

\title[]{Realizing Galois representations in abelian varieties by specialization}
\author[]{Arvind Suresh}

\begin{abstract}
    We give some positive answers to the following problem: Given a field $K$ and a continuous Galois representation $\rho:G_K \to GL_n(\mathbf{Q})$, construct an abelian variety $J/K$ of small dimension such that $\rho$ is a sub-representation of the natural $G_K$-representation on $J(\overline{K}) \otimes_{\mathbf{Z}} \mathbf{Q}$. We prove that if $K$ is Hilbertian of characteristic different from $2$, then for any sufficiently large integer $g$ (depending on $\rho$) we can find infinitely many absolutely simple $g$-dimensional abelian varieties which realize $\rho$. We outline also a method of twisting a given symmetric construction of curves with many rational points to instead produce curves with closed points of large degree, and in this context we give a unified treatment of constructions of Mestre--Shioda and Liu--Lorenzini. The main results are obtained by applying a natural generalization of N\'eron's Specialization Theorem.
\end{abstract}

\date{\today}
\maketitle

\tableofcontents

\section{Introduction} 

Let $K$ be a field with separable closure $\kbar$, and let $G_K$ denote the absolute Galois group $\gal(\kbar/K)$. By a \defi{Galois module}, we mean a $\qgk$-module $V$, finite-dimensional over $\Q$, for which the associated Galois representation $\rho_V: G_K \too GL(V)$ is continuous, i.e., has kernel of finite index (such representations usually go by the name ``{Artin representations}'' in the literature). An abelian variety $J/K$ is said to \defi{realize} $V$ if $J(\kbar)_{\Q} \ceq J(\kbar) \otimes_{\Z} \Q$ contains a $\qgk$-submodule isomorphic to $V$ (equivalently, if $\rho_V$ can be realized as a sub-representation of the natural Galois representation on $J(\kbar)_{\Q}$). Now, our goal in this article is to exhibit positive solutions to the following problem.
\begin{quote}
    \textbf{The Realization Problem.} \textit{Given a Galois module $V$, produce an abelian variety $J/K$ which realizes $V$, with $\dim J$ small relative to $\dim_{\Q} V$.}
\end{quote}

Rohrlich~\cite{rohrlich-allreps}*{Proposition~17} showed that if $K$ is any field and $E/K$ is an elliptic curve with a $K$-rational point of infinite order, then {any} Galois module is realized by a suitable twisted form $J/K$ of the abelian variety $E^n/K$, for some $n\geqslant 1$. 
This leaves open the question of whether every Galois module is realized by an \emph{absolutely simple} abelian variety over $K$.  \Cref{thm:all-reps} below gives an affirmative answer to this question under fairly general assumption on $K$. Below, and for the rest of the article, by ``infinitely many objects over $K$'' we mean ``infinitely many objects over $K$ that are pair-wise non-isomorphic over $\kbar$.'' 

\begin{thm}[proved in \ref{proof-all reps}]\label{thm:all-reps}
Let $K$ be a Hilbertian field (see \ref{emp:thin-sets}) of characteristic different from $2$, and let $V$ be a Galois module. Then, for any sufficiently large integer $g$ (the implied lower bound depending on $V$), there exist infinitely many $g$-dimensional absolutely simple abelian varieties over $K$ which realize $V$.
\end{thm}

\begin{emp}\label{emp:l-functions}
\textbf{$L$-functions.} Let $K$ be a number field, and $J/K$ an abelian variety. Given a continuous complex representation $\rho$ of $G_K$, one can form the {twisted $L$-function} $L(J,\rho,s)$ (see \cite{rohrlich-vanishing}*{Section~2}).
The \defi{Generalized BSD Conjecture} (see \cite{rohrlich-allreps}*{Page 311, (0.1)} and \cite{rohrlich-vanishing}*{Proposition~2}) asserts that the multiplicity $m_{\rho}$ of $\rho$ in the space $J(\kbar)\otimes \C$ equals the order of vanishing $\textup{ord}_{s=1} L(J,\rho,s)$. Assuming this conjecture, \Cref{thm:all-reps} implies the following statement:

\textit{If $K$ is a number field and $\rho$ is a complex {  not necessarily irreducible} $G_K$-representation, then for any sufficiently large integer $g$, there exist infinitely many $g$-dimensional absolutely simple abelian varieties $J/K$ such that $L(J,\rho,s)$ vanishes at $s=1$.}

If $\rho$ is a self-dual representation, then the weaker \defi{Parity Conjecture}~\cite{rohrlich-allreps}*{Page 311, (0.2)} asserts that $(-1)^{m_{\rho}}$ equals the global root number $w(J,\rho)$ (see \cite{rohrlich-allreps}*{Section~4} for the definition). There is a rich literature devoted to computing formulas for these root numbers and giving (conditional) realization results for representations $\rho$ by showing that $w(J,\rho)=-1$ (e.g.~\cites{rohrlich-allreps,howe-growth,dokchitser-root-nonabelian,sabitova-twisted-root,bisatt-explicit-root}, to name but a few).
\end{emp}

\begin{emp}\label{emp:ranks-intro}
\textbf{Mordell--Weil ranks.} Our principle motivation for studying the Realization Problem lies in its connection to Mordell--Weil ranks. 
Suppose $K$ is a number field, so that $J(K)$ is a finitely generated abelian group for any abelian variety $J/K$ (Mordell--Weil Theorem). The \defi{Mordell--Weil rank} of $J(K)$ is the integer $\rk J(K) \ceq \dim_{\Q} J(K)_{\Q}$.
\begin{enumerate}[(a)]
    \item (\textit{Large rank}) \label{emp:ranks-intro:triv}
    In the case of trivial Galois action, the Realization Problem amounts to asking for abelian varieties of large rank and small dimension. Indeed, if $V = \triv_{G_K}^n$, where $n$ is a positive integer and $\triv_{G_K}$ denotes the one-dimensional \defi{trivial module}, then $J/K$ realizes $V$ if and only if $\rk J(K)\geqslant n$. 
    \item (\textit{Rank growth}) \label{emp:ranks-intro:V(L/K)}
    Realizing non-trivial Galois modules has implications for the {growth} in Mordell--Weil rank with extensions of the ground field. Indeed, given a finite extension $L/K$, let $\V{L}{K}$ denote the Galois module determined (up to isomorphism) by the identity
    \begin{equation*}
        \ind{G_K}{G_L}\triv_{G_L} \cong \triv_{G_K} \oplus \V{L}{K}.
    \end{equation*}
    If $J/K$ realizes $\V{L}{K}$, then the rank of $J$ grows in \emph{every} intermediate extension of $L/K$, i.e., $\rk J(M) > \rk J(F)$ for any pair $F$ and $M$ such that $K \subseteq F \subsetneq M \subseteq L$ (see \Cref{lem:realize-rank}).  
    \end{enumerate}
\end{emp}

\begin{emp}\label{emp:V(Omega/K)}
\textbf{The Galois module $\V{\Omega}{K}$.}
\Cref{thm:all-reps} is a corollary to  \Cref{thm:main} below, which concerns a class of Galois modules that we now define. Recall that a \defi{\fetk} $\Omega$ is a ring of the form $\Omega = L_1 \times \dotsb \times L_r$, where the $L_i$ are (not necessarily distinct) finite separable extensions of $K$. The \defi{degree} of $\Omega$ is the integer $\dim_K \Omega$. 
Associated to such an $\Omega$, we define the Galois module 
\begin{equation*}
    \V{\Omega}{K} \ceq \triv_{G_K}^{r-1}\oplus \V{L_1}{K}\oplus \dotsb \oplus \V{L_r}{K}.
\end{equation*}
For example, if $L/K$ is a finite separable extension and $\Omega = L^r$, then $\V{\Omega}{K} = \triv_{G_K}^{r-1} \oplus \V{L}{K}^r$. In particular, if $K$ is a number field, then an abelian variety $J/K$ which realizes $\V{L^r}{K}$ satisfies $\rk J(K) \geqslant r-1$ and $\rk J(L) - \rk J(K) \geqslant r$ (cf. \ref{emp:ranks-intro}).  
\end{emp}


\begin{thm}[proved in \ref{proof-main}]\label{thm:main}
Let $K$ be a Hilbertian field of characteristic different from $2$, and $\Omega$ a finite \'etale $K$-algebra of degree $n\geqslant 1$. Then, for any positive integer {  $g$ which satisfies $4g+6 \geqslant n$}, 
there exist infinitely many genus $g$ hyperelliptic curves $X/K$ such that: 
\begin{enumerate}[(a),leftmargin=1cm,itemindent=0pt]
    \item The Jacobian $J_X/K$ realizes the $(n-1)$-dimensional Galois module $\V{\Omega}{K}$. \label{thm:main:realize}
    \item If $g\neq 2$ or $\chr K\neq 3$, then $J_X/K$ is absolutely simple.\label{thm:main:simple}
\end{enumerate}
In particular, if $n\leqslant 10$, then the curves $X/K$ can be chosen to be elliptic curves (with pair-wise distinct $j$-invariants). 
\end{thm}

\begin{exmp}\label{exmp:main}
We illustrate \Cref{thm:main} in the case of elliptic curves with a few examples.
\begin{enumerate}[(a)]
    \item \label{exmp:main:rohrlich-matsuno} Rohrlich~\cite{rohrlich} proved by explicit construction that if $L/K$ is an extension of number fields with $[L:K]\leqslant 9$, then there exists an elliptic curve over $K$ which realizes $\V{L}{K}$. Matsuno~\cite{matsuno-zp} extended his construction to produce, given a \fetk of $\Omega$ of degree at most $9$, infinitely many elliptic curves which realize $\V{\Omega}{K}$ (although he states it only for a single extension $L/K$). \Cref{thm:main} improves this by allowing $\deg \Omega = 10$, and by allowing $K$ to be any Hilbertian field.
    \item \label{exmp:main:cubic} Let $L/K$ be a cubic extension of number fields. Taking $\Omega = L^3\times K$, \Cref{thm:main} gives infinitely many $E/K$ which realize $\V{\Omega}{K}= \V{L}{K}^3 \oplus \triv_{G_K}^3$, and which therefore satisfy $\rk E(K)\geqslant 3$ and $\rk E(L) - \rk E(K) \geqslant 3$. We view this as a complement to results in the literature (e.g.~\cites{dokchister-cubic,vanishing-cubic,kozuma-cubic,lemke-thorne}) which \emph{fix} the curve $E/K$ and show that the rank (algebraic and analytic) goes up in infinitely many cubic extensions $L/K$.
    \item \label{exmp:main:deg-5} The example (b) above (minus a factor of $\triv_{\gg{K}}$) can also be achieved by Matsuno's construction~\cite{matsuno-zp}. On the other hand, if $L/K$ is a degree $5$ extension of number fields, \Cref{thm:main} gives elliptic curves $E/K$ which realize $\V{L}{K}^2 \oplus \triv_{G_K}$, and hence satisfy $\rk E(L) - \rk E(K) \geqslant 2$, which is not achieved by \cite{matsuno-zp}.
    \item \label{exmp:main:A5} Let $L/\Q$ be the quintic extension obtained by adjoining a root of the polynomial $x^5 + 6x^4 + 19x^3 + 25x^2 + 11x + 2$. The splitting field of this extension has Galois group isomorphic to the alternating group $A_5$ (see~\cite{rohrlich-allreps}*{Page~312}), and $\V{L}{\Q} \otimes_{\Q} \C$ corresponds to the unique degree $4$ irreducible complex representation of $A_5$. It follows from~\cite{rohrlich-allreps}*{Proposition~B} and the Generalized BSD Conjecture (see \ref{emp:l-functions}) that for any elliptic curve $E/\Q$, $\V{L}{\Q}$ occurs in $E(\qbar)_{\Q}$ with \emph{even} multiplicity. Taking $\Omega \ceq L \times \Q^5$, \Cref{thm:main} gives infinitely many elliptic curves $E/\Q$ which realize $\V{\Omega}{\Q} = \V{L}{\Q} \oplus \triv_{G_{\Q}}^5$, and by the above discussion, these curves (conjecturally) realize the $13$-dimensional Galois module $\V{L}{\Q}^2 \oplus \triv_{G_{\Q}}^5$. 
\end{enumerate}
\end{exmp}

\begin{emp}\label{emp:V-wt-Omega}
\textbf{The Galois module $\V{\wt{\Omega}}{\Omega}$.}
Our last realization result is \Cref{thm:quad} below; for a class of Galois modules $V$, it gives abelian varieties realizing $V$ which are of smaller dimension than those given by \Cref{thm:main}. 
Given a \fetk $\Omega$ and a finite \'etale $\Omega$-algebra $\wt{\Omega}$, we define $\V{\wt{\Omega}}{\Omega}$ to be the Galois module defined up to isomorphism by the identity
\begin{equation}
    \V{\wt{\Omega}}{K} \cong \V{\wt{\Omega}}{\Omega} \oplus \V{\Omega}{K}.
\end{equation}
For example, suppose $\wt{\Omega}$ is a \defi{quadratic extension} of $\Omega\ceq L_1\times \dotsb \times L_r$, i.e., $\wt{\Omega} = C_1 \times \dotsb \times C_r$, with each $C_i = L_i^2$ or $C_i = \wt{L}_i$ for some separable quadratic extension $\wt{L}_i/L_i$. Then (see \Cref{lem:V-wt-Omega-identity})
\begin{equation}\label{eq:V-wt-Omega-identity}
    \V{\wt{\Omega}}{\Omega} = \ind{G_K}{G_{L_1}} \V{C_1}{L_1} \oplus \dotsb \oplus  \ind{G_K}{G_{L_r}} \V{C_r}{L_r}.
\end{equation}
\end{emp}
 

\begin{thm}[proved in \ref{proof-quad}]\label{thm:quad}
Let $K$ be a Hilbertian field of characteristic different from $2$, $\Omega$ a finite \'etale $K$-algebra of degree $n\geqslant 1$, and $\wt{\Omega}$ a quadratic extension of $\Omega$. 
Then, for any positive integer {  $g$ which satisfies $4g+4 \geqslant n$}, 
there exist infinitely many genus $g$ hyperelliptic curves $X/K$ such that:
\begin{enumerate}[(a)]
    \item \label{thm:quad:realize}The Jacobian $J_X/K$ realizes the $n$-dimensional Galois module $\V{\wt{\Omega}}{\Omega}$. 
    \item \label{thm:quad:simple}If $g\neq 2$ or $\chr K\neq 3$, then $J_X/K$ is absolutely simple.
\end{enumerate}
In particular, if $n\leqslant 8$, then the curves $X/K$ can be chosen to be elliptic curves (with pair-wise distinct $j$-invariants).
\end{thm}

\begin{exmp}\label{exmp:matsuno-Z2}
We again illustrate \Cref{thm:quad} in the case of elliptic curves with a few examples. For simplicity, we fix extensions of number fields $\wt{L}/L/K$, with $\wt{L}/L$ quadratic. To make the translation to Mordell--Weil ranks, we note (see \Cref{lem:realize-rank}) that if an abelian variety $J/K$ realizes $\ind{G_K}{G_L} \V{\wt{L}}{L}$, then $\rk J(\wt{L}) > \rk J(L)$. 
\begin{enumerate}[(a)]
    \item If $[L:K]\leqslant 5$, {  so that $[\wt{L}:K] \leqslant 10$,} then \Cref{thm:main} yields elliptic curves $E/K$ which realize $\V{\wt{L}}{K}$, and hence, also its submodule $\ind{G_K}{G_L} \V{\wt{L}}{L}$ (cf. \Cref{lem:realize-rank}). If $[L:K]=6,7,$ or $8$, then \Cref{thm:quad} gives elliptic curves which realize $\ind{G_K}{G_L} \V{\wt{L}}{L}$, a result which is not achieved by \Cref{thm:main}. As a special case, if $K=\Q$, and $L$ and $\wt{L}$ denote the maximal real subfield of $\Q(\zeta_{32})$ and $\Q(\zeta_{64})$, respectively, we recover a result of Matsuno~\cite{matsunoz2}*{Theorem~1.1} which gives infinitely many elliptic curves $E/\Q$ with $\rk E(\wt{L}) > \rk E(L)$. 
    \item Suppose $L/K$ is a quartic extension. Taking $\Omega \ceq L^2$ and $\wt{\Omega} \ceq \wt{L}^2$, \Cref{thm:quad} gives infinitely many elliptic curves $E/K$ which realize $(\ind{G_K}{G_L} \V{\wt{L}}{L})^2$, and hence, satisfy $\rk E(\wt{L}) - \rk E(L) \geqslant 2$. 
\end{enumerate}
\end{exmp}

\begin{emp}
\textbf{Acknowledgements.} This work is part of the author's doctoral dissertation at the University of Georgia. The author is indebted to his PhD advisor, Dino Lorenzini, for his thoughtful and useful feedback on earlier versions of the paper. The author is grateful to Daniel Litt for helping clarify some of the ideas in \Cref{sec:twisting}. The author would also like to acknowledge and thank the mathematics department at UGA for providing an environment conducive to research.
\end{emp}

\section{Outline of method and paper}

In this section, we outline the method underlying our results, and along the way we describe the layout of the paper. 

\begin{emp}\label{emp:conventions}
\textbf{Conventions.}
Throughout this article, $K$ denotes an infinite field of characteristic different from $2$. For a $K$-scheme $X$ (resp. $K$-algebra $R$) and field extension $L/K$, we write $X_L$ (resp. $R_L$) to denote the base-change $X \times_K L$ (resp. tensor product $R \otimes_K L$). By a \defi{$K$-variety} we mean a geometrically integral, separated scheme of finite type over $K$. By a \defi{nice curve over $K$}, we mean a smooth, projective, $K$-variety of dimension $1$. We frequently conflate  effective divisors on a nice curve $X$ with proper, closed subschemes of $X$. 

By a \defi{hyperelliptic curve} over $K$, we mean a nice curve $X/K$ which admits a degree $2$ $K$-morphism to the projective line $\p^1_K$\footnote{Note that, contrary to the convention in the literature, we do not make any assumption on the genus of $X$.}. Given a non-constant polynomial $\ell(x) \in K[x]$ which is not a square in $K[x]$, we write $X/K: \, y^2 = \ell(x)$ to mean that $X/K$ is the smooth proper model of the affine curve defined by the above equation.

We denote tuples of indeterminates by boldfont, e.g. $\mbf{p}$ denotes a tuple $p_1,\dotsc,p_n$; in this case, we denote the affine space $\spec K[\mbf{p}]$ by $\A_K(\mbf{p})$, and if $K$ is fixed, by $\A(\mbf{p})$. 
\end{emp}


\begin{emp}\label{emp:outline-single}
\textbf{Basic strategy.} 
Given a nice curve $X/K$ and effective divisor $D \subset X$, we associate a Galois submodule $\W{D} \subset J_X(\kbar)_{\Q}$, which we regard as the ``contribution'' to $J_X(\kbar)_{\Q}$ made by $D$:
\begin{equation*}
    \W{D} \ceq \textup{span}_{\Q} \{ [P - Q] \mid P,Q \in D({\kbar})\} \subset J_X(\kbar)_{\Q}.
\end{equation*}
If we choose some point $O \in D({\kbar})$ as a base-point for the Abel-Jacobi map $j:X_{\kbar} \inj (J_X)_{\kbar}$, then $\W{D}$ is the $\qgk$-submodule of $J_X(\kbar)_{\Q}$ spanned by the points in  $j(D({\kbar})) \subset J_X({\kbar})$.

Now, in order to realize a given Galois module $V$ in an abelian variety over $K$, we first choose a \fetk $\Omega$ such that $V \subset \V{\Omega}{K}$ (see \Cref{lem:any V in V(A/K)}). Then, we construct a nice curve $X/K$ of small (positive) genus and an effective divisor $D\subset X$ such that 
\begin{enumerate}[(i)]
    \item $D$ is a \defi{divisor of type $\Omega$}, i.e., $D \cong \spec \Omega$, and
    \item $\dim_{\Q} \W{D} = n - 1$ (here $n=\deg \Omega$), i.e.,  writing $D(\kbar) = \{O,P_1,\dotsc,P_{n-1}\}$, the divisor classes $[P_i - O], i=1,\dotsc, n - 1$ are linearly independent in $J(\kbar)$. 
\end{enumerate}
The Galois module $\W{D}$ is then isomorphic to $\V{\Omega}{K}$ (see \Cref{lem:V(D/K)} \ref{lem:V(D/K):surj}), so $J_X/K$ realizes $\V{\Omega}{K}$, and hence, also $V$.
\end{emp}

\begin{emp}\label{emp:outline-spec}
\textbf{Specialization.} To apply the strategy \ref{emp:outline-single} for an infinite family of curves, we proceed by specialization. That is, we take a rational function field $\ff/K$ and construct a hyperelliptic curve $\X'/\ff$ of small (positive) genus equipped with an effective divisor $\D'$ of type $\Omega_{\ff} \ceq \Omega \otimes_K \ff$ such that $W(\D')\cong \V{\Omega}{K}$ as  $\qgk$-modules\footnote{To make sense of this statement, we need to first view $W(\D')$ as a $\qgk$-module; this is explained in \ref{emp:specialization-intro}.}. Then, we appeal to  \Cref{lem:specialization}, a simple extension of the classical Specialization Theorem of N\'eron, to conclude that $\X'/\ff$ can be specialized down to an infinite family of curves over $K$ with Jacobian realizing $V$. Finally, we use a criterion of Zarhin~\cite{zarhinSimple} (see \Cref{sec:zarhin}) to obtain the desired simplicity statements.
\end{emp}

\begin{emp}\label{emp:LLC}
\textbf{The Liu--Lorenzini Construction.} 
We will apply the specialization method to the construction of Liu--Lorenzini in~\cite{newpts}, which proceeds as follows. Suppose we are given a positive integer $g$ and a \fetk $\Omega$ of degree $n \ceq 4g+6$. Let $\{\alpha_1,\dotsc,\alpha_n\}$ be a $K$-basis for $\Omega$, and let $\mbf{z}$ denote the tuple $z_1,\dotsc,z_n$. Identify $\A(\mbf{z}) \ceq \spec K[\mbf{z}]$ with the \defi{Weil Restriction} $\res{\Omega/K} \A^1_{\Omega}$ by associating to each $K$-point $t \ceq (t_1,\dotsc,t_n)$ of $\A(\mbf{z})$ the element $\alpha_t \ceq \sum_{i=1}^n t_i\alpha_i  \in \Omega$. That is, $\alpha_t$ is the specialization at $t$ of the element
\begin{equation*}
    \alpha \ceq \sum_{i=1}^n z_i\alpha_i \in \Omega[\mbf{z}].
\end{equation*}
Since $\Omega[\mbf{z}]/K[\mbf{z}]$ is \'etale of degree $n$, $\alpha$ admits a characteristic polynomial $\chi(x) = x^n + f_{n-1}(\mbf{z})x^{n-1} + \dotsb + f_0(\mbf{z}) \in K[\mbf{z}]$. The coefficients $m_i \ceq f_i(\mbf{z})$, for $i=0,\dotsc,n-1$, turn out to be algebraically independent over $K$, and the inclusion $K[\mbf{m}] \subset K[\mbf{z}]$ defines a finite $K$-morphism of affine spaces 
\begin{equation*}
    \chi : \res{\Omega/K} \A^1_{\Omega} \too \A(\mbf{m}).
\end{equation*}
Associating to each point  $(a_0,\dotsc,a_{n-1}) \in \A(\mbf{m})$ the polynomial $x^n + a_{n-1}x^{n-1} + \dotsb + a_0$, the map $\chi$ acquires a functorial interpretation-- on $K$-points it gives the \defi{characteristic polynomial map} for $\Omega/K$, sending each $\alpha \in \Omega$ to its characteristic polynomial $\chi_{\alpha}(x)\in K[x]$. Thus, $\chi$ coincides with the composition of the morphisms $f$ and $\mu$ defined in \cite{newpts}*{2.8}, the construction of which lies at the heart of their method.

Now, let $\ff/\mm$ denote the function field extension determined by $\chi$ (i.e., $\ff = K(\mbf{z})$ and $\mm = K(\mbf{m})$), and view $m(x) = x^n + m_{n-1}x^{n-1} + \dotsb + m_0$ as a polynomial in $\mm[x]$. The fundamental algebraic lemma (\Cref{lem:sqroot}) furnishes polynomials $h(x)$ and $\ell(x) \in \mm[x]$, with $\ell(x)$ separable of degree $2g+2$, such that $m(x) = h(x)^2 - \ell(x)$. Since $m(x)$ is the characteristic polynomial of $\alpha$ (viewed as an element of $\Omega_{\ff}$), we have $m(\alpha) = 0$, and hence, $$h(\alpha)^2 = \ell(\alpha) \in \Omega_{\ff}.$$ Thus, $\X/\mm: y^2 = \ell(x)$ is a genus $g$ hyperelliptic curve endowed with a closed immersion $\spec \Omega_{\ff} \inj \X_{\ff}$ which in affine coordinates is the point $(\alpha,h(\alpha)) \in \X(\Omega_{\ff})$. In fact, if we put $\D \ceq \divzero{y-h(x)} \subset \X$, then $\D_{\ff}$ the image of this closed immersion, so $\D_{\ff}$ is a divisor of type $\Omega_{\ff}$ on $\X_{\ff}$. It turns out that $W(\D_{\ff}) \cong \V{\Omega}{K}$ as a $\qgk$-module; we explain this below. 
\end{emp}

\begin{emp}\label{emp:MSC}
\textbf{The Mestre--Shioda Construction.}
Mestre ~\cites{mestre11,mestre12} used \Cref{lem:sqroot} to construct genus $g$ curves with $n \ceq 4g+6$ points as follows. Let $K[\mbf{m}] \inj K[\mbf{u}] \ceq K[u_1,\dotsc,u_n]$ be the injection which gives the factorization $m(x) = (x-u_1) \dotsb (x-u_n) \in K[\mbf{u}][x]$. The associated $K$-morphism $$\chi':\A(\mbf{u}) \too \A(\mbf{m})$$ is a quotient for the natural permutation action of the symmetric group $S_n$ on $\A(\mbf{u})$. Putting $\kk \ceq K(\mbf{u})$, we find that the genus $g$ curve $\X/\mm$ from \ref{emp:LLC} has the points $P_i \ceq (u_i,h(u_i)) \in \X(\kk)$, for $i=1,\dotsc,n$. In fact, $\D_{\kk} \subset \X_{\kk}$ is the union of these points, and Shioda~\cite{shiodasymmetry} showed (in the terminology of \ref{emp:outline-single}) that $\dim_{\Q} W(\D_{\kk}) = n-1$.   

The key point now (forming the basis for this paper) is that the construction in \ref{emp:LLC} is transformed, upon extending scalars from $K$ to the splitting field $L$ of $\Omega$, into the construction of Mestre--Shioda. To see this, let $\{e_1,\dotsc,e_n\}$ denote the standard $K$-basis for $K^n$, and observe that if view $\A(\mbf{u})$ as the Weil Restriction $\res{K^n/K} \A^1_{K^n}$ (as in \ref{emp:LLC}) using the element $$\alpha' \ceq \sum_{i=1}^n u_ie_i \in K^n[\mbf{u}],$$ then $\chi'$ can be understood as the characteristic polynomial map for $K^n/K$. Thus, choosing any $L$-algebra isomorphism $\phi: \Omega_L \iso L^n$, we can define an $L$-isomorphism 
\begin{equation*}
    \Phi: (\res{\Omega/K} \A^1_{\Omega})_L \iso \A^n_L(\mbf{u})
\end{equation*}
which gives $\phi$ on $L$-valued points. By construction, $\Phi$ is a $\A^n_L(\mbf{m})$-morphism, so the isomorphism $\kk_L \iso \ff_L$ induced by $\Phi$ is an $\mm_L$-algebra isomorphism. Thus, it extends to an isomorphism $\X(\kk_L) \iso \X(\ff_L)$ sending the points $(u_i,h(u_i))$ to the geometric points in the support of $\D_{\ff} \subset \X_{\ff}$. Shioda's result then implies that  $W(\D_{\ff}) \cong \V{\Omega}{K}$, as desired. 
\end{emp}



\begin{emp}\label{emp:outline-layout}
\textbf{Twisting constructions.}
The work~\cite{newpts} contains, to our knowledge, the only method for systematically constructing curves $X/K$  with $\Omega$-valued points for a given $\Omega/K$. The main novelty (if any) in the present article is that, with a view towards generalizing~\cite{newpts} and its connection to the Mestre--Shioda Construction, we develop a way to \emph{twist} constructions of curves with many rational points to get curves with $\Omega$-valued points (see~\Cref{prop:twisting-C}). This approach evinces the observation that producing curves with $\Omega$-valued points is \emph{geometrically} the same as producing curves with $(\deg \Omega)$-many rational points. As proof-of-concept, we give a unified treatment of the Mestre--Shioda and Liu--Lorenzini constructions; we show that the latter is obtained by twisting the former (see \ref{emp:coordinatize}).  In \Cref{sec:moduli}, we discuss the connection between the method and twisted forms of the moduli space $M_{g,n}$. 
\end{emp}

\begin{rem}\label{rem:elliptic}
As we mentioned in \Cref{exmp:main} \ref{exmp:main:rohrlich-matsuno}, \Cref{thm:main} extends the results of Rohrlich and Matsuno to allow realization of $\V{\Omega}{K}$ in elliptic curves when $\deg \Omega = 10$. It is natural to ask about the next step, namely, given $\Omega/K$ of degree $11$, can one produce an elliptic curve $E/K$ which realizes $\V{\Omega}{K}$? The twisting procedure of this article could, in principle, be applied to known families of elliptic curves with rank $\geqslant 10$ to realize $\V{\Omega}{K}$ with $\deg \Omega = 11$. However, to our knowledge, the existing constructions in the literature do not admit enough symmetry (cf. \ref{emp:symmetry}) to yield elliptic curves containing a degree $11$ closed point. We explain this limitation in greater detail in \Cref{exmp:shioda-11}, and we formulate a related question about the moduli space $M_{1,11}$ in \Cref{quest:M1-11}.
\end{rem}

\section{Preliminaries}\label{sec:prelims}



\begin{emp}\label{emp:G-sets}
\textbf{$G$-sets.}  Let $G$ be a profinite group. By a \defi{$G$-set}, we mean a discrete set $S$ equipped with a continuous action of $G$. We write $\Q[S]$ to denote the $\qg$-module $\oplus_{s \in S} \,\Q \cdot s$ generated by $S$.\footnote{All $\qg$-modules considered in this article will be \emph{left} $\qg$-modules.} We write $\triv_G$ for the one-dimensional trivial $\qg$-module. There is a natural \defi{degree map} $\deg : \Q[S] \too \triv_{G}$, defined by  $\sum_{s \in S} a_s \cdot s \mtoo \sum_{s \in S} a_s$, whose kernel we denote by $\Q[S]_0$. When $S$ is a finite set, there is a decomposition  $\Q[S] \cong \Q[S]_0 \oplus \triv_G$ (Maschke's Theorem). 

If $D$ is an effective divisor on a curve $X/K$, then $D(\kbar)$ is a $G_K$-set. The degree map $\Q[D(\kbar)] \too \triv_{G_K}$ is induced by the usual degree map on divisors, and we have
\begin{equation*}
    \Q[D(\kbar)]_0 \ceq \textup{span}_{\Q}\{ P- Q \mid P,Q \in D(\kbar)\} \, \subset (\Div^0 X_{\kbar}) \otimes_\Z {\Q}.
\end{equation*}
\end{emp}

\begin{emp}\label{emp:fetk}
\textbf{Divisors of type $\Omega$.}
Let $\Omega$ be a finite \'etale $K$-algebra. We say that a field extension $L/K$ \defi{splits} $\Omega/K$ if $\Omega_L$ is isomorphic to $L^n$ as an $L$-algebra; the smallest such field is the \defi{splitting field} of $\Omega/K$. We write $\Et(n,L/K)$ for the set of degree $n$ \'etale $K$-algebras (up to $K$-isomorphism) split by $L/K$.  
We call an effective divisor $D$ on a curve $X/K$ a divisor of \defi{of type $\Omega$} if $D \cong \spec \Omega$. In this case, if $L/K$ splits $D/K$ (i.e., it splits $\Omega/K$), then we have $D_L \cong \sqcup_{i=1}^n (\spec L)$, i.e., $D_L$ consists of $n$ $L$-rational points of $X_L$ (here $n=\deg D = \deg \Omega$).

Recall from \ref{emp:V(Omega/K)} and \ref{emp:outline-single} the Galois modules $\V{\Omega}{K}$ and $\W{D}$, respectively.
\end{emp}

\begin{lem}\label{lem:V(D/K)}
If $D$ is a divisor of type $\Omega$ on a nice curve $X/K$, then 
\begin{enumerate}[(a),itemindent=0pt,leftmargin=1cm]
    \item \label{lem:V(D/K):Omega}$\Q[D(\kbar)]_0 \cong \V{\Omega}{K}$, and
    \item \label{lem:V(D/K):surj} the natural map $E \mapsto [E]$ defines a surjection of Galois modules $\V{\Omega}{K} \surj \W{D}$, which is an isomorphism if  $\dim_{\Q}  \W{D} = \deg D - 1.$
\end{enumerate}
\end{lem}

\begin{proof}
Part \ref{lem:V(D/K):surj} follows from \ref{lem:V(D/K):Omega} and the definitions of $\Q[D(\kbar)]_0$ and $\W{D}$. For \ref{lem:V(D/K):Omega}, write $\Omega = L_1 \times \dotsb \times L_r$, so that $D \cong \spec L_1 \sqcup \dotsb \sqcup \spec L_r$. The natural isomorphism of $G_K$-sets
\begin{equation*}
    D(\kbar) \cong \bigsqcup_{i=1}^r (\spec L_i)(\kbar) \cong \bigsqcup_{i=1}^r G_K/G_{L_i}
\end{equation*}
gives rise to an isomorphism of Galois modules (in the notation of \ref{emp:G-sets}):
\begin{align*}
    \triv_{G_K} \oplus \Q[D(\kbar)]_0 & \cong \Q[D({\kbar})]\\
     & \cong \Q[G_K/G_{L_1}] \oplus \dotsb \oplus \Q[G_K/G_{L_r}]\\
     & \cong \ind{G_K}{G_{L_1}} \triv_{G_{L_1}} \oplus \dotsb \oplus \ind{G_K}{G_{L_r}} \triv_{G_{L_r}}\\
     & \cong (\triv_{G_K} \oplus \V{L_1}{K}) \oplus \dotsb \oplus (\triv_{G_K}\oplus \V{L_r}{K})\\
     & \cong \triv_{G_K} \oplus \V{\Omega}{K}.
\end{align*}
It follows that $\Q[D(\kbar)]_0 \cong \V{\Omega}{K}$. 
\end{proof}

The \defi{splitting field} of a Galois module $V$ is the fixed field of the kernel of the associated representation $\rho_V : G_K \too GL(V)$. The splitting field of $\V{\Omega}{K}$ coincides with the splitting field of $\Omega/K$. For a finite Galois extension $L/K$, we will often regard a $\Q \glk$-module $V$ as a $\qgk$-module by conflating $V$ with the inflation $\textup{Infl}_{\gg{L/K}}^{G_K} V$.

\begin{lem}\label{lem:any V in V(A/K)}
Suppose $V$ is a Galois module with splitting field $L \supsetneq K$, which is a direct sum of $r$ irreducible submodules. Put $\Omega \ceq L^r$. Then, $V$ is isomorphic to a submodule of $\V{\Omega}{K}$.
\end{lem}

\begin{proof}
Put $G \ceq \gal(L/K)$. Since $\ind{G_L}{G_K} \triv_{G_L} = \triv_{G_K} \oplus \V{L}{K}$ is a free $\qg$-module of rank $1$, every non-trivial irreducible $\qg$-module is isomorphic to a submodule of $\V{L}{K}$. Thus, we can write $V = V_1 \oplus \dotsb \oplus V_k \oplus \triv_{G_K}^{r-k}$, for some $k\leqslant r$, and with each $V_i$ a submodule of $\V{L}{K}$. Note that the assumption $L\neq K$ forces $k\geqslant 1$. We have $\V{\Omega}{K} \cong \V{L}{K}^r \oplus \triv_{G_K}^{r-1}$, and it is clear that $V$ is a submodule of $\V{\Omega}{K}$.
\end{proof}

Recall from \ref{emp:V-wt-Omega} that, if $\Omega$ is a \fetk and $\wt{\Omega}$ is a finite \'etale $\Omega$-algebra, then $\V{\wt{\Omega}}{\Omega}$ denotes the Galois module defined by the identity $\V{\wt{\Omega}}{K} \cong \V{\wt{\Omega}}{\Omega} \oplus \V{\Omega}{K}$.

\begin{lem}\label{lem:V-wt-Omega-identity}
Let $\Omega\ceq L_1 \times \dotsb \times L_r$ be a \fetk of degree $n$, and let $\wt{\Omega}$ be a finite \'etale $\Omega$-algebra of degree $e$. Write $\wt{\Omega} = C_1 \times \dotsb \times C_r$, with each $C_i$ a finite \'etale $L_i$-algebra of degree $e$. 
Then, the Galois module $\V{\wt{\Omega}}{\Omega}$ is of dimension $en-n$ and satisfies
\begin{equation*}
    \V{\wt{\Omega}}{\Omega} \cong \ind{G_K}{G_{L_1}} \V{C_1}{L_1} \oplus \dotsb \oplus  \ind{G_K}{G_{L_r}} \V{C_r}{L_r}.
\end{equation*}
\end{lem}

\begin{proof}
For each $i=1,\dotsc,r$, we write $C_i$ as a product of fields $C_i = M_{i1} \times \dotsb \times M_{is_i}$, so that $\wt{\Omega} = \prod_{i=1}^r \prod_{j=1}^{s_i} M_{ij}$. Then, the desired isomorphism follows from the computation:
\begin{align*}
    \triv_{G_K} \oplus \V{\wt{\Omega}}{K} & \cong \bigoplus_{i=1}^r \bigoplus_{j=1}^{s_i} \ind{G_K}{G_{M_{ij}}} \triv_{G_{M_{ij}}}\\
    & \cong \bigoplus_{i=1}^r \ind{G_K}{G_{L_i}} \bigg( \bigoplus_{j=1}^{s_i} \ind{G_{L_i}}{G_{M_{ij}}} \triv_{G_{M_{ij}}} \bigg) \\
    & \cong \bigoplus_{i=1}^r \ind{G_K}{G_{L_i}} \big( \triv_{G_{L_i}} \oplus \V{C_i}{L_i} \big)\\
    & \cong \bigoplus_{i=1}^r \ind{G_K}{G_{L_i}} \triv_{G_{L_i}} \, \bigoplus_{i=1}^r \ind{G_K}{G_{L_i}} \V{C_i}{L_i}\\
    & \cong \triv_{G_K} \oplus \V{\Omega}{K} \oplus \bigoplus_{i=1}^r \ind{G_K}{G_{L_i}} \V{C_i}{L_i}.\qedhere
\end{align*}
\end{proof}

\begin{lem}\label{lem:realize-rank}
Let $L/K$ be a finite, separable extension, and $J/K$ an abelian variety.
\begin{enumerate}[(a),itemindent=0pt,leftmargin=1cm]
    \item \label{lem:realize-rank:sum}Suppose $F$ is an intermediate field in $L/K$ and $V_1,\dotsc,V_r$ are non-zero $\Q G_K$-submodules of $\ind{G_K}{G_F} \V{L}{F}$. If $J/K$ realizes $V_1 \oplus \dotsb \oplus V_r$, then $\rk J(L) - \rk J(F) \geqslant r$.
    \item \label{lem:realize-rank:grow}If $J$ realizes $\V{L}{K}$, then $\rk J(M) > \rk J(F)$ for any pair of fields $M$ and $F$ such that $K \subseteq F \subsetneq M \subseteq L$.
\end{enumerate}
\end{lem}

\begin{proof}
\begin{enumerate}[(a)]
    \item The key point is that $\rk J(L) > \rk J(F)$ if and only if there is a vector $v \in J(L)_{\Q}$ which is not fixed by $G_F$, or equivalently, a vector $v \in J(\kbar)_{\Q}$ which is fixed by $G_L$ but \emph{not} by $G_F$. Since $G_F$ acts non-trivially on any non-zero vector of $\ind{G_K}{G_F} \V{L}{F}$, it suffices to produce vectors $v_i \in V_i,\, i=1,\dotsc,r$, which are fixed by $G_L$.
    To that end, fix $i \in \{1,\dotsc,r\}$, and choose a $\qgk$-module surjection $W\ceq \ind{G_K}{G_L} \triv_{G_L} \surj V_i$; such a surjection exists because $V_i$ is a submodule of $W$ (see the proof of (b) below), and hence also a quotient of $W$. Note that $W\cong \Q[G_K/G_L]$, and the trivial coset is a vector in $\Q[G_K/G_L]$ whose stabilizer is precisely $G_L$. The image $v_i$ of this vector in $V_i$ is fixed by $G_L$, as desired. 
    \item Given intermediate fields $K \subset F \subsetneq M \subset L$, we have an isomorphism of Galois modules 
    \begin{equation*}
        \ind{G_K}{G_L} \triv_{G_L} \cong \ind{G_K}{G_M} \V{L}{M} \oplus \ind{G_K}{G_F} \V{M}{F} \oplus \V{F}{K} \oplus \triv_{G_K}.
    \end{equation*}
    Since the left-hand side is isomorphic to $\triv_{G_K} \oplus \V{L}{K}$, we have an isomorphism
    \begin{equation*}
        \V{L}{K} \cong \ind{G_K}{G_M} \V{L}{M} \oplus \ind{G_K}{G_F} \V{M}{F} \oplus \V{F}{K},
    \end{equation*}
    and part \ref{lem:realize-rank:grow} now follows from \ref{lem:realize-rank:sum}. \qedhere
\end{enumerate}
\end{proof}

\section{The specialization method for Galois modules}\label{sec:specialization}
\begin{emp}\label{emp:specialization-intro}
In this section, $L/K$ denotes a finite Galois extension, and $\ff/K$ denotes a \defi{function field}, i.e., the function field of a positive-dimensional $K$-variety. We write $\ff_L$ for the field $\ff \otimes_K L$. For any abelian variety $\J/\ff$, the group $\J(\ff_L)$ is a $G_{\ff_L/\ff}$-module, and via the canonical isomorphism between $\glk$ and $G_{\ff_L/\ff}$, we view $\J(\ff_L)$ as a $\glk$-module. 
We say that $\J/\ff$ realizes a given $\Q \glk$-module $V$ if $\J(\ff_{L})_{\Q}$ contains a $\Q \glk$-submodule isomorphic to $V$.

The main result of this section is \Cref{lem:specialization}, which says that an abelian variety over $\ff$ which realizes a $\Q \glk$-module $V$ can, under certain standard hypotheses, be specialized to give abelian varieties over $K$ which realize $V$. The proof consists in reducing to the case of the trivial module $V=\triv_{G_K}^n$, which is nothing but the classical Specialization Theorem of N\'eron~\cite{neronhilbert}*{Th\'eor\`eme~6} (stated in \Cref{thm:neron}). This reduction step relies on \Cref{lem:thin-over-ext}.
\end{emp}

\begin{emp}\label{emp:thin-sets}
\textbf{Thin sets.} Recall from~\cite{serremw}*{Page~121, Section~9.5} or~\cite{CTrankjump}*{Section 2} that for a $K$-variety $U$, a subset $T \subset U(K)$ is \defi{thin in $U$} if there is a $K$-scheme $X$ and a  generically finite $K$-morphism $f:X \too U$, admitting no rational section $U \DashedArrow X$, such that $f(X(K)) \supset T$. 
We say that $K$ is \defi{Hilbertian} if $U(K)$ is not thin for  any $K$-rational variety $U$.  It is known (see~\cite{serremw}*{Section~9.5, Page~129}) that 
(1) for an arbitrary field $F$, the function field $F(t)$ is Hilbertian, and (2) any finitely generated field is Hilbertian. 
\end{emp}

\begin{emp}\label{emp:specialization}
\textbf{Specialization.} Let $U$ be a smooth $K$-variety with function field $\ff$, and let $\calj/U$ be an abelian scheme with generic fiber $\J/\ff$. For a rational point $P \in U(K)$, the fiber $\calj_P/K$ is called the \defi{specialization} of $\J$ at $P$.
Let $\eta:\spec \ff \too U$ be the inclusion of the generic point, and let $P \in U(K)$ be a rational point. Given a section $s:U \too \calj$, we can base-change along $\eta$ to get a point $s_{\eta} : \spec \ff \too \J$, and similarly, along $P$ to get a point $s_P:\spec K \too \calj_P$.\footnote{Taking these base-changes corresponds to intersecting the image of $s$ with $\J \subset \calj$ and $\calj_P \subset \calj$, respectively.} The maps $(-)_{\eta}:\calj(U)\too \J(\ff)$ and $(-)_P:\calj(U) \too \calj_P(K)$ are group homomorphisms, and in fact, $(-)_{\eta}$ is an isomorphism by~\cite{movinglemma}*{Proposition~6.2}. The \defi{specialization map} $\sigma_P:\J(\ff) \too \calj_P(K)$ is the composition
\begin{equation*}
    \begin{tikzcd}
    \J(\ff) \arrow{r}{{(-)_{\eta}\inv}}[swap]{\sim} & \calj(U) \arrow[r, "(-)_{P}"] & \calj_P(K).
    \end{tikzcd}
\end{equation*}
In this situation, N\'eron's Specialization Theorem says the following.
\end{emp}

\begin{thm}[{\cite{serremw}*{Section~11.1}}] \label{thm:neron}
Assume that the group $\J(\ff)$ is finitely generated. Then, the set $\{\, P \in U(K) \mid  \sigma_P : \J(\ff) \longrightarrow \calj_P(K) \textup{ is not injective } \}$ is thin in $U$. 
\end{thm}

\begin{lem}[{\cite{serremw}*{Page~128}}]\label{lem:thin-over-ext}
If a subset $T \subset {U}_L(L)$ is thin in ${U}_L$, then $T \cap U(K)$ is thin in ${U}$.
\end{lem}

\begin{lem}[Specialization for Galois modules]\label{lem:specialization}
Suppose $U$ is a smooth $K$-variety with function field $\ff/K$, and $\calj /{U}$ is an abelian scheme with generic fiber $\J/\ff$. Assume that the group $\J(\ff_L)$ is finitely generated. 
\begin{enumerate}[(a),itemindent=0pt,leftmargin=1cm]
    \item For any point $P \in {U}(K)$, regarded as a point of ${U}_L(L)$, the specialization map $\sigma_P: \J(\ff_L) \too \calj_P(L)$ is a $\glk$-module homomorphism.\label{lem:specialization:galois}
    \item The set $T \ceq \{\, P \in {U}(K) \mid \sigma_P : \J(\ff_L) \too \calj_P(L) \textup{ is not injective } \}$
    is thin in $U$.\label{lem:specialization:thin}
\end{enumerate}
\end{lem}


\begin{proof}
\begin{enumerate}[(a)]
    \item We use the natural $\Z \glk$-isomorphism $\calj(U_L) = \J(\ff_L)$ to identify points $s : \spec \ff_L \too \J$ with sections $s: U_L \too \calj$. Let $g\in \glk$, and let $g^*:\spec L \iso \spec L$ be the associated $K$-isomorphism. Given a point $P \in U(K)$, we have a commutative diagrams (of $U$-schemes on the left, and of $K$-schemes on the right):
    \begin{equation*}
        \begin{tikzcd}
        U_L \arrow[dr, "s \circ g^*"] \arrow[dd, "{(\id , g^*)}",swap] & & &\spec L \arrow[dr, "\sigma_P(s\circ g^*)"] \arrow[dd, "g^*",swap] \\
        &  \calj & & &  \calj_P\\
        U_L \arrow[ur, "s"'] & & & \spec L \arrow[ur, "\sigma_P(s)"'].  
        \end{tikzcd}
    \end{equation*}
    The triangle on the right is the base-change of the triangle on the left via $P:\spec K \too U$, so we have $\sigma_P(s \circ g^*) = \sigma_P(s)\circ g^*\, \in \textup{Mor}_K(\spec L, \calj_P)$, which establishes \ref{lem:specialization:galois}.
    \item Let $T' \ceq \{ P \in U_L(L) \mid \sigma_P : \J(\ff_L) \too \calj_P(L) \textup{ is not injective} \},$ so that $T = T' \cap U(K)$. \Cref{thm:neron} says that $T'$ is thin in $U_L$, and \Cref{lem:thin-over-ext} then implies that $T$ is thin in $U$.\qedhere
\end{enumerate}
\end{proof}


\begin{rem}\label{rem:matsuno shioda}
We are aware of two antecedents in the literature in which authors have constructed given \emph{non-trivial} Galois modules in abelian varieties by specialization. The first is the construction of Rohrlich--Matsuno~\cite{matsuno-zp} (see \Cref{exmp:main} \ref{exmp:main:rohrlich-matsuno}). 
The second is the beautiful construction of Shioda~\cites{shioda-family,shiodaMW}, in which he used his theory of Mordell--Weil lattices of elliptic surfaces to show that if $L/\Q$ is a Galois extension with group isomorphic to the Weyl group $W(E_6),W(E_7),$ or $W(E_8)$, then there exist infinitely many elliptic curves $E/\Q$ which realize a certain $8$-dimensional irreducible $\Q \glk$-module. This case lies outside the scope of Theorems \ref{thm:main} and \ref{thm:quad}.

Both of these constructions rely on Silverman's strengthening~\cite{silverman-specialization} of the N\'eron Specialization theorem; this replaces the ingredient \Cref{lem:thin-over-ext} to yield the specialization outcome of \Cref{lem:specialization} in the case when $K$ is a number field and the base of the abelian scheme is an open subset of $\p^1_K$. We remark also that this strategy was noted by Rohrlich in~\cite{rohrlich-family}.
\end{rem}

\section{A brief review of Galois cohomology}

\begin{emp}\label{emp:galois-coh-intro}
In this section, we review some basic facts from (non-abelian) Galois cohomology, largely following~\cite{serregaloiscoh}*{Chapters~I.5} and~\cite{Qpoints}*{Sections~4.4 and 4.5}. If $g$ is an automorphism of a ring $R$, we write $g^*$ for the induced automorphism of $\spec R$, and we identify $\aut R$ with $\aut(\spec R)$ via the natural isomorphism $g \mtoo (g^*)\inv$.
By an \defi{action} of a group $G$ on a scheme $X$, we mean a homomorphism $s:G \too \aut(X)$. If $X$ and $X'$ are schemes endowed with $G$-actions $s$ and $s'$, respectively, then we call a morphism $f:X \too X'$ \defi{$G$-equivariant} if $s'(g) \circ f = f\circ s(g)$ for all $g\in G$, i.e., $f$ commutes with the $G$-actions on $X$ and $X'$. 

Throughout, we fix a finite Galois extension $L/K$. For any scheme $X/K$, the natural $\glk$-action on $X_L$ is defined by $\sigma \mapsto (\id,\sigma)$; we abuse notation and denote the automorphism $(\id,\sigma)$ by $\sigma$. This $\glk$-action is \defi{semi-linear}, i.e., $X_L \too \spec L$ is $\glk$-equivariant. The theory of Galois Descent provides the following converse (for a proof, see~\cite{Qpoints}*{Corollary~4.4.6} and use the equivalent definitions of ``descent datum'' in~\cite{Qpoints}*{Propositions 4.4.2 (i) and 4.4.4 (i)}).
\end{emp}

\begin{lem}[Galois Descent]\label{lem:galois-descent}
Let $Y/L$ be affine or quasi-projective, and let $s:\glk \too \aut_K(Y)$ be a semi-linear action. Then, the quotient $X \ceq Y/s(\glk)$ exists and there is a $\glk$-equivariant isomorphism $\phi:Y \iso X_L$  (i.e., $\phi \circ s(\sigma) = \sigma \circ \phi$ for all $\sigma \in \glk$). 
\end{lem}

\begin{exmp}\label{exmp:descent-affine}
If $Y = \spec S$ is affine, then a semi-linear $\glk$-action on $Y$ corresponds to a semi-linear $\glk$-action on the $L$-algebra $S$, and we have $X = \spec R$, where $R$ denotes the ring of invariants $S^{s(\glk)}$. \emph{Speiser's Lemma}~\cite{gille-szamuely}*{Lemma~2.3.8} says that the map $R \otimes_K L \to S$ induced by the inclusions $R \subset S$ and $L \subset S$ is an $L$-algebra isomorphism. Passing to spectra, this defines the desired $\glk$-equivariant isomorphism $\phi: Y \iso X_L$. 
\end{exmp}


\begin{emp}\label{emp:GK-groups}
\textbf{$\glk$-groups.} By a  \defi{$\glk$-group}, we mean a group $G$ equipped with an action of $\glk$, denoted $(\sigma,g) \mtoo \prsup{\sigma}{g}$. As usual, $Z^1(\glk,G)$ denotes the set of  \defi{$1$-cocycles of $\glk$ in $G$} (i.e., set maps $a:\glk \too G$, denoted $\sigma \mapsto a_{\sigma}$, which satisfy the cocycle condition $a_{\sigma \tau} = a_{\sigma} \prsup{\sigma}a_{\tau}$), and $H^1(\glk,G)$ denotes the quotient set $Z^1(\glk,G)/\sim$, where $a \sim b$ if $a'_{\sigma} = g\inv a_{\sigma} \prsup{\sigma}{g}$ for some $g\in G$ and all $\sigma \in \glk$. If the $\glk$-action on $G$ is trivial then $Z^1(\glk,G)$ is the set of group homomorphisms $\Hom(\glk,G)$, and $H^1(\glk,G)$ is the quotient set $\Hom(\glk,G)/\sim$, where $a\sim b$ if $b_{\sigma} = g\inv a_{\sigma} g$ for some $g\in G$ and all $\sigma \in \glk$. 
\end{emp}

\begin{emp}\label{emp:twisting}
\textbf{Twisting.} 
Let $X$ be an affine or quasi-projective $K$-scheme. By an \defi{$L/K$-twist} of  $X$ we mean a pair $(X',\phi)$, where $X'$ is a $K$-scheme and $\phi: X_L \iso X'_L$ is an $L$-isomorphism; at times we may suppress the isomorphism $\phi$ for convenience. An $L/K$-twist of a $K$-algebra $R$ is a pair $(R',\phi)$, where $R'$ is a $K$-algebra and $\phi:R'_L \iso R_L$ is an $L$-algebra isomorphism. 

For any $K$-scheme $Y$, the group $\glk$ acts naturally on the set $\mor_L(X_L,Y_L)$ by $$\prsup{\sigma}{f} \ceq \sigma \circ f \circ \sigma \inv, \qquad \textup{ for } \sigma \in \glk \textup{ and } f \in \mor_L(X_L,Y_L).$$ In particular, this action makes $G \ceq \aut_L(X_L)$ into a $\glk$-group.  Any cocycle $a \in Z^1(\glk,G)$ gives rise to the \defi{twisted-by-$a$} semi-linear $\glk$-action
\begin{equation*}
    s_a:\glk \too \aut_K(X_L), \quad \sigma \mtoo s_a(\sigma) \ceq a_{\sigma} \circ \sigma. 
\end{equation*}
The \defi{twist of $X$ by $a$} is the quotient $\prsup{a}{X} \ceq X_L/s_a(\glk)$. (If $X = \spec R$ is affine, then $\prsup{a}{X} = \spec \prsup{a}{R}$, where $\prsup{a}{R}$ is the ring of invariants $R^{s_a(\glk)}$.) 
\Cref{lem:galois-descent} says that there is an $L$-isomorphism $\phi_a:X_L \iso \prsup{a}{X}_L$ which for all $\sigma \in \glk$ satisfies $\phi_a \circ s_a(\sigma) = \sigma \circ \phi_a$. 

The map $a \mapsto (\prsup{a}{X},\phi_a)$ is a bijection from $Z^1(\glk,G)$ to the set of $L/K$-twists of $X$. The inverse sends a twist $(X',\phi)$ to the \emph{difference cocycle}\footnote{It is so named because it encodes the difference between the natural $\glk$-actions on $X_L$ and $X'_L$.} $a\in Z^1(\glk,G)$ defined by $a_{\sigma} \ceq \phi\inv \circ \prsup{\sigma}{\phi}$. 
This induces the well-known bijection (see~\cite{Qpoints}*{Theorem~4.5.2})
\begin{equation*}
    H^1(\glk,\aut_L(X_L)) \iso \dfrac{\{\textup{$L/K$-twists of $X$}\}}{\textup{$K$-isomorphism}}, \qquad [a] \mtoo [\prsup{a}{X}].
\end{equation*}
\end{emp}


\begin{exmp}\label{emp:Et-n-L/K}
If $X = \spec K^n$, then $X_L = \spec L^n$ and $\aut_K(X) = \aut_L(X_L) = S_n$, the symmetric group $\aut_{\textup{Sets}}\{1,\dotsc,n\}$ with trivial Galois action. Identifying $\aut_K(K^n)$ with $\aut_K(X)$ (as in \ref{emp:galois-coh-intro}), each $g \in \aut_K(X)$ corresponds to the automorphism of $K^n$ defined on the standard basis $\{e_1,\dotsc,e_n\}$ by $e_i \mtoo e_{g(i)}$. Identifying $\Et(n,L/K)$ (cf. \ref{emp:fetk}) with the set of $L/K$-twists of $\spec K^n$ (modulo $K$-isomorphism), we get a bijection
\begin{equation*}
    H^1(\glk,S_n) \iso \Et(n,L/K), \qquad [a] \mtoo [\prsup{a}{(K^n)}].
\end{equation*}
\end{exmp}

\begin{emp}\label{emp:kalg}
\textbf{Functoriality.} 
Let $G$ be a group, and let $\sch{K}{G}$ denote the category of affine or quasi-projective $K$-schemes  endowed with a $G$-action $G \too \aut_K(X)$ (morphisms in $\sch{K}{G}$ are $G$-equivariant $K$-morphisms). Such an action induces a $\glk$-group homomorphism $f:G \too \aut_L(X_L)$ (viewing $G$ as a $\glk$-group with trivial action). For a cocycle $a\in \Hom(\glk,G)$, we abuse notation and write $\prsup{a}{X}$ for the twist $\prsup{f \circ a}{X}$, and we again call this the twist of $X$ by $a$. Let us recall here some basic facts about the twisting operation.
\begin{enumerate}[(i)]
    \item \label{emp:kalg:i} Twisting by $a\in Z^1(\glk,G) = \Hom(\glk,G)$ defines a functor $\prsup{a}{(-)}$ from $\sch{K}{G}$ to the category of $K$-schemes-- we can {twist} any morphism $\pi:X \too Y$ in $\sch{K}{G}$ to get a $K$-morphism $\prsup{a}{\pi}:\prsup{a}{X} \too \prsup{a}{Y}$ such that the associated $L$-isomorphisms $\phi_a: X_L \iso \prsup{a}{X}_L$ and $\psi_a: Y_L \iso \prsup{a}{Y}_L$ satisfy $\prsup{a}{\pi}_L \circ \phi_a = \psi_a \circ \pi_L : X_L \too \prsup{a}{Y}_L$. If moreover the $G$-action on $Y$ is trivial then $\prsup{a}{Y} = Y$ and $\psi_a$ is the identity, so we have a commutative diagram
    \begin{equation*}
    \begin{tikzcd}
        X_L \arrow[rr, "{\resizebox{0.45cm}{0.1cm}{$\sim$}}"', "\phi_a"] \arrow[dr, "\pi_L"'] && \prsup{a}{X}_L \arrow[dl, "\prsup{a}{\pi}_L"] \\
        & Y_L.
    \end{tikzcd}
    \end{equation*}
    \item \label{emp:kalg:ii}If $X$ and $Y$ are objects in $\sch{K}{G}$, then we view $X \times Y$ again as an object in $\sch{K}{G}$ by endowing it with the natural \emph{diagonal action} of $G$, where $g\in G$ acts by the automorphism $(g,g)$. Then, the functor $\prsup{a}(-)$ commutes with taking products-- we have a $K$-isomorphism
    \begin{equation*}
        \prsup{a}{(X\times Y)} \iso \prsup{a}{X} \times \prsup{a}{Y}.
    \end{equation*}
\end{enumerate}
\end{emp}

\section{Twisting a  construction}\label{sec:twisting}

\begin{emp}\label{emp:constructions}
\textbf{Constructions.}
Let $K$ be a field, and $\Omega$ a finite \'etale $K$-algebra of degree $n\geqslant 1$. By a \defi{genus $g$ construction for $\Omega/K$}, we mean a tuple of data $\scrc \ceq (\kk,\X',\D',\ul{P})$, where
\begin{enumerate}[(i)]
    \item $\kk/K$ is a function field,\label{emp:construction:kk}
    \item $\X'/\kk$ is a nice curve of genus $g$, and\label{emp:construction:X}
    \item $\D' \subset \X'$ is a divisor of type $\Omega_{\kk} \ceq \Omega \otimes_K \kk$. 
    \item $\ul{P} : \spec \Omega_{\kk} \iso \D'$ is a $\kk$-isomorphism. 
\end{enumerate}
The natural case (the case of ``trivial Galois action'') is when $\Omega = K^n$. Since $\spec (K^n)_{\kk} \cong \spec \kk^n = \sqcup_{i=1}^n \spec \kk$, giving a construction $(\kk,\X',\D',\ul{P})$ for $K^n/K$ amounts to giving a nice curve $\X'/\kk$ with distinct marked points $P_1,\dotsc,P_n \in \X'(\kk)$. Indeed, $n$ such points together define a $\kk$-isomorphism $\ul{P}: \spec (K^n)_{\kk}\iso \D'$, where $\D'$ is the union of (the images of) the points. Conversely, given the isomorphism $\ul{P}$, the $\kk$-points of $\spec {\kk}^n$ (each corresponding to the projection maps $p_i: \kk^n \surj \kk$) define $n$ marked points $P_i \ceq \ul{P}\circ p_i^* \in \X'(\kk)$. 
\end{emp}

\begin{emp}\label{emp:symmetry}
\textbf{Symmetry.}
Suppose $\scrc \ceq (\kk,\X',\D',\ul{P})$ is a construction for $K^n/K$.  Let $G$ be a finite subgroup of $\kk$ with fixed field $\mm$, so that $G$ is the Galois group $\gg{\kk/\mm}$. We say that $\scrc$ \defi{admits symmetry by $G$} if either of the following equivalent conditions\footnote{The equivalence of conditions (i) and (ii) follows by Galois Descent (\Cref{lem:galois-descent}).} hold:
\begin{enumerate}[(i)]
    \item \label{item:symmetry:i} The $G$-action on $\spec \kk$ extends to a semi-linear $G$-action on $\X'$ and its divisor $\D'$.
    \item \label{item:symmetry:ii} There exists a curve $\X/\mm$ and divisor $\D \subset \X$ such that $\X' = \X_{\kk}$ and $\D' = \D_{\kk}$.
\end{enumerate}
\end{emp}

\begin{exmp}\label{exmp:joshi}
The space $S$ of homogeneous ternary cubic forms over $\Q$ is isomorphic to $\p^9_{\Q}$. It is well-known that there is an open subscheme $U \subset (\p^2_{\Q})^9$ (parametrizing nine points in general position) such that for any point $(P_1,\dotsc,P_9) \in U({\Q})$, there is a unique cubic curve over $\Q$ passing through all the $P_i$'s. So, if we define
    \begin{equation*}
        \calx' \ceq \{(f,Q,P_1,\dotsc,P_9) \mid f(Q) = f(P_1) = \dotsb = f(P_9) = 0 \} \subset S \times \p^2_{\Q} \times U,
    \end{equation*}
then the projection $\pi: \calx' \too U$ admits nine disjoint sections $\calp_i:U \too \calx'$. Putting $\kk \ceq {\Q}(U)$, the generic fiber $\X'/\kk$ of $\pi$ is a smooth cubic curve endowed with nine $\kk$-points $P_1,\dotsc,P_9$ (each $P_i$ is the image of the generic point of $V$ under $\calp_i$). So, if $\D'$ denotes the union of these points, then $\scrc \ceq (\kk,\X',\D',\ul{P})$ is a genus one construction for ${\Q}^9/{\Q}$. 

The natural $S_9$-action on $V$ extends to an action on $\calx'$ which preserves the image of the sections $\calp_1,\dotsc,\calp_9$. Restricting to the generic point of $V$, we obtain an $S_9$-action on $\kk$ that extends to an action on $\X'$ fixing $\D'$, so $\scrc$ has symmetry by $S_9$. The fixed field $\mm$ of $S_9$ is the function field of $B \ceq V/S_9$, which is itself an open subscheme of the symmetric power $\sym^9(\p^2_{\Q}) \ceq (\p^2_{\Q})^9/S_9$. The curve $\X'$ descends to a smooth cubic $\X/\mm$, and $\D'$ to a degree $9$ closed point $\D \in \X$.
\end{exmp}



\begin{emp}\label{emp:G-action}
Fix a construction $\scrc \ceq (\kk,\X_{\kk},\D_{\kk},\ul{P})$ with symmetry by a finite group $G \subset \aut_K(\kk)$ (as in \ref{emp:symmetry}). Choose a degree $n$ polynomial $m(x) \in \mm[x]$ such that  $\D = \spec \mm[x]/m(x)$. Then, $\D_{\kk} = \spec \kk[x]/m(x)$, and the $G$-set $\D(\kk)$ is naturally isomorphic to $\mf{R}$, the set of roots of $m(x)$ in $\kk$. The isomorphism $\ul{P}$ is defined on coordinate rings by
\begin{equation*}
    \ul{P}^{\#} : \kk[x]/m(x) \iso (K^n)_\kk, \qquad b(x) \mtoo \sum_{i=1}^n e_i \otimes b(r_i),
\end{equation*}
where $\{e_1,\dotsc,e_n\}$ is the standard basis of $K^n$ and $\mf{R} = \{ r_1,\dotsc,r_n \}$\footnote{Note: This ordering of the $G$-set $\mf{R}$ is determined by the isomorphism $\ul{P}$.}. The natural $G$-action on $\mf{R}$ determines a permutation representation $G \too S_n = \aut_{\textup{Sets}}\{1,\dotsc,n\}$ such that
\begin{equation}\label{eq:G-to-Sn}
    g(r_i) = r_{g(i)},  \textup{ for all } g\in G \textup{ and } i=1,\dotsc,n.
\end{equation}
Now, let $G$ act on $K^n$ via this homomorphism (i.e., $g$ acts by $e_i \mapsto e_{g(i)}$)\footnote{If we identify $(K^n)_{\kk}$ with $\kk^n$ via the map $\sum_i e_i \otimes b_i \mapsto (b_i)_i$, then $g\in G$ acts by $(b_i) \mapsto (b_{g\inv(i)})$.}. Using that $(\spec K^n)_{\kk} \ceq (\spec K^n) \times_K {\kk}$ is canonically isomorphic to $\spec (K^n)_{\kk}$, we have the following.
\end{emp}

\begin{lem}\label{lem:G-action}
The $\kk$-isomorphism $\ul{P}: (\spec K^n)_{\kk} \iso \D_{\kk}$ is $G$-equivariant, where $G$ acts on $(\spec K^n)_{\kk}$ (resp. $\D_{\kk}$) by the diagonal action $g \mapsto (g,g)$ (resp. natural action $g \mapsto (\id,g)$). 
\end{lem}

\begin{proof}
The automorphism $(\id,g) \in \aut_{\mm}(\D_{\kk})$ corresponds to the automorphism $g$ of $\kk[x]/m(x)$ which acts on coefficients, i.e., sends an element $b(x) = \sum_i b_ix^i$ to $(g\cdot b)(x) \ceq \sum_{i} g(b_i)x^i$. Similarly, $(g,g) \in \aut_{\mm}((\spec K^n)_{\kk})$ corresponds  to the automorphism
    \begin{equation*}
        g\otimes g \in \aut_{\mm}(K^n \otimes_K \kk), \qquad \sum_{i=1}^n e_i \otimes \alpha_i \mtoo \sum_{i=1}^n e_{g(i)} \otimes g(\alpha_i).
    \end{equation*}
    Proving that $\ul{P}$ is $G$-equivariant amounts to showing that for all $g\in G$, we have
    \begin{equation}\label{eq:g-otimes-g-hom}
        (g \otimes g) \circ \ul{P}^{\#} = \ul{P}^{\#} \circ g \in \Hom (\kk[x]/m(x), K^n \otimes_K \kk). 
    \end{equation}
    Let $b(x)\in \kk[x]/m(x)$ and $g\in G$. Writing $c(x) = (g\cdot b)(x)$, we see that $\ul{P}^{\#} \circ g$ sends $b(x)$ to $\sum_{i=1}^n e_i \otimes c(r_i)$. Since $g(b(r_i)) = c(g(r_i)) = c(r_{g(i)})$, \eqref{eq:g-otimes-g-hom} follows because
    \begin{equation*}
    \begin{tikzcd}
         b(x) \arrow[r, mapsto, "\ul{P}^{\#}"] & {\displaystyle\sum_{i=1}^n} e_i \otimes b(r_i) \arrow[r, mapsto, "{g \otimes g}"] & {\displaystyle\sum_{i=1}^n} e_{g(i)} \otimes g(b(r_i)) = {\displaystyle\sum_{i=1}^n} e_i \otimes c(r_i) 
    \end{tikzcd} \qedhere
    \end{equation*}
\end{proof}

\begin{emp}\label{emp:Et-C}
\textbf{The set $\Et({\scrc},L/K)$.}
The representation $G \too S_n$ in \ref{emp:G-action} induces, for any finite Galois extension $L/K$, a map of pointed sets $H^1(\glk,G) \too H^1(\glk,S_n) = \Et(n,L/K)$ (cf. \ref{emp:Et-n-L/K}). We define $\Et(\scrc,L/K)$ to be the image of this map, i.e., 
\begin{equation*}
    \Et(\scrc,L/K) \ceq \{ \prsup{a}{(K^n)} \mid a \in \Hom(\glk,G) \} \subset \Et(n,L/K).
\end{equation*}
For example, if $G = S_n$ (as in \Cref{exmp:joshi} and the Mestre--Shioda Construction in \ref{emp:MSC}), then $\Et(\scrc,L/K) = \Et(n,L/K)$. 
Now, with this setup, we have the following punchline.
\end{emp}

\begin{thm}[Twisting a construction]\label{prop:twisting-C}
Let $a\in \Hom(\glk,G)$ be a cocycle. Put $\Omega \ceq \prsup{a}{(K^n)} \in \Et(\scrc,L/K)$ and $\ff \ceq \prsup{a}{\kk}$. Then, we have the following.
\begin{enumerate}[(a)]
    \item \label{prop:twisting-C-isom-mm}$\ff/K$ is a function field containing $\mm$, and $\spec \kk_L \iso \spec \ff_L$\footnote{We are referring here to the $L$-isomorphism coming from the definition of $L/K$-twists (cf. \ref{emp:twisting}).} is an $\mm_L$-isomorphism.
    \item \label{prop:twisting-C-isom-twist}Twisting $\ul{P}$ by the cocyle $a$ gives an $\ff$-isomorphism
        \begin{equation}\label{eq:prsupa-theta}
            \prsup{a}{\ul{P}} : (\spec \Omega)_{\ff} \iso \D_\ff. 
        \end{equation}
    Thus, $\prsup{a}{\scrc} \ceq (\ff,\X_{\ff},\D_{\ff},\prsup{a}{\ul{P}})$ is a genus $g$ construction for $\Omega/K$ (the \defi{twist of $\scrc$ by $a$}).
    \item \label{prop:twisting-C-isom-dim} $W(\D_{\ff})$ is isomorphic to a  $\qgk$-submodule of $\V{\Omega}{K}$. Moreover, $\dim_{\Q} W(\D_{\ff}) = \dim_{\Q} W(\D),$
    so if $W(\D)$ is of maximal dimension (namely, $n-1$), then $W(\D_{\ff}) \cong \V{\Omega}{K}$. 
    \end{enumerate}
\end{thm}

\begin{proof}
\begin{enumerate}[(a)]
    \item The morphism $\spec \kk \too \spec \mm$ is $G$-equivariant (with $G$ acting trivially on $\spec \mm$), so \ref{prop:twisting-C-isom-mm} follows from the functoriality of the twisting operation (see \ref{emp:kalg} \ref{emp:kalg:i}). 
    \item \Cref{lem:G-action} says that the $G$-action on $(\spec K^n)_{\kk}$ (which makes $\ul{P}$ $G$-equivariant) is the diagonal action. So, the product $(\spec \Omega)_{\ff} = (\spec \Omega) \times_K (\spec \ff)$ of the respective twists is $K$-isomorphic to the twist $\prsup{a}{(\spec K^n) \times_K (\spec \kk)}$ (cf. \ref{emp:kalg} \ref{emp:kalg:ii}). Applying the functor $\prsup{a}(-)$ to $\ul{P}: (\spec K^n)_{\kk} \iso \D_{\kk}$ yields the isomorphism \eqref{eq:prsupa-theta}. This proves \ref{prop:twisting-C-isom-twist}.
    \item Since $\D(\ff_L)$ is a $\glk$-set isomorphic to $\Hom_K(\Omega,L)$, \Cref{lem:V(D/K)} implies that $W(\D_{\ff}) \subset \J(\ff_L)_{\Q}$ is a $\Q \glk$-submodule of $\V{\Omega}{K}$, with equality if and only if $\dim_{\Q} \W{\D_{\ff}} = n-1$. The key point now is that $\W{\D_{\ff}}$, viewed as a $\Q G_{\ff}$-module, is the restriction $\textup{Res}_{G_{\ff}}^{G_{\mm}} W(\D)$. So, we have $\dim_{\Q} \W{\D_{\ff}}= \dim_{\Q} \W{\D}$, and part \ref{prop:twisting-C-isom-dim} follows.\qedhere
\end{enumerate}
\end{proof}

\begin{exmp}\label{exmp:easy-symmetry}
Take any nice curve $\X$ over a function field $\mm/K$. For a general divisor $\D \subset \X$ which is \'etale of degree $n$ over $\mm$, one expects that the splitting field $\kk$ of $\D$ is a function field over $K$. In this case, if choose a $\kk$-isomorphism $\ul{P}: \spec (K^n)_{\kk} \iso \D_{\kk}$, then  $(\kk,\X_{\kk},\D_{\kk},\ul{P})$ is a construction for $K^n/K$ with symmetry by $G \ceq \gg{\kk/\mm}$. Note that $G$ acts faithfully $\D(\kk)$, so the associated representation $G \inj S_n$ is injective. If moreover $\D/\mm$ is irreducible (i.e. $\D$ is a closed point of $\X$) then $G$ is a \emph{transitive subgroup} of $S_n$. In this case, an \'etale $K$-algebra $\Omega$ of degree $n$ lies in $\Et(\scrc,L/K)$ if and only if the Galois group of its splitting field is a subgroup of $G$. 
\end{exmp}

\begin{exmp}\label{exmp:shioda-11}
A fundamental limitation of the twisting method is that the symmetry in a given construction $\scrc$ may not be ``big enough'' so that $\Et(\scrc,L/K)$ contains a given $\Omega/K$.  For example, if $F/\Q$ is a given degree $11$ field extension and we wish to construct an elliptic curve realizing $\V{F}{\Q}$, we might attempt to twist the constructions in~\cite{neron-N1} (elucidated in~\cite{shioda-elliptic-11}) and ~\cite{mestre11}, both of which are constructions of elliptic curves over $\Q$ with rank at least $11$. However, as we describe now, neither of these approaches would work.

\begin{enumerate}[(a)]
	\item N\'eron constructs a family $E/U$ of elliptic surfaces, with $U/\Q$ a $6$-dimensional rational variety, together with an elliptic curve $\Gamma/\Q(E)$ having $11$ independent rational points $Q_1,\dotsc,Q_8,M_1,M_2,M_3$ (cf. ~\cite{shioda-elliptic-11}*{Lemma~2}). This defines a genus $1$ construction $\scrc$ for $\Q^{11}/\Q$, but if it has symmetry by some finite group $G$ then the representation $G \too S_{11}$ is not transitive. Indeed, condition (ii) on~\cite{shioda-elliptic-11}*{Page~110}, which is used in defining $U$, and hence, also $E$, implies that under the action of any subgroup $G \subset \Q(E)$ on the points of $\Gamma$, the points $\{Q_1,Q_2,Q_3\}$ can never lie in the same orbit as those in $\{Q_4,\dotsc,Q_8\}$. Thus, $\Et(\scrc,L/\Q)$ does not contain any number field of degree $11$. 
	\item In~\cite{mestre11}, Mestre constructs a rational $\Q$-variety $U$ and  an elliptic curve $E/\Q(U)$  with $12$ marked points generating rank $11$ over $\Q(U)$. This defines a genus $1$ construction $\scrc$ for $\Q^{12}/\Q$ with symmetry by the alternating group $A_4$ (see~\cite{elkiesthreelectures}*{Pages~4 and 5} for an exposition). Note that if $\Omega = F_1 \times \dotsb \times F_r$ is of the form $\prsup{a}{(K^{12})}$ for some cocycle $a\in \Hom(\glk,A_4)$, then the Galois group of the splitting field of each $F_i$ is a subgroup of $A_4$. Thus, no $\Omega \in \Et(\scrc,L/\Q)$ can contain a degree $11$ number field as a summand. 
\end{enumerate} 
\end{exmp}


\begin{emp}\label{emp:linear}
\textbf{Linear actions.} 
We will apply \Cref{prop:twisting-C} in a situation where $\kk/K$ is a rational function field and the $G$-action on $\kk$ is \defi{linear}, by which we mean that $\kk = K(x_1,\dotsc,x_n)$ for some elements $x_i$ such that the $G$-action on $\kk$ restricts to a linear representation on the vector space $\textup{span}_K\{x_1,\dotsc,x_n\}$. The next lemma will allow us to apply the specialization method to each of the twisted constructions $\prsup{a}{\scrc}$.  
\end{emp}

\begin{lem}\label{lem:hilbert 90}
In \Cref{prop:twisting-C}, if the $G$-action on $\kk$ is linear, then every twist $\ff \ceq \prsup{a}{\kk}$ appearing in \ref{prop:twisting-C} is again a rational function field over $K$.
\end{lem}

\begin{proof}
By assumption, we have inclusions of $\glk$-groups $G \subset GL_n(L) \subset \aut_L({\kk}_L)$. These induce maps of pointed sets $H^1(\glk,G) \too H^1(\glk,GL_n(L)) \too H^1(\glk,\aut_L({\kk}_L))$. By Hilbert 90, we have $H^1(\glk,GL_n(L)) = \{1\}$, so any class $(a) \in H^1(\glk,G)$ is mapped to the class of the trivial twist in $H^1(\glk,\aut_L({\kk}_L))$, i.e.,  $\prsup{a}{\kk}$ is $K$-isomorphic to $\kk$.
\end{proof}

\begin{rem}
\Cref{prop:twisting-C} is a special case (the ``split'' case) of a more general procedure whereby one can twist a construction for some $\Omega \in \Et(n,L/K)$ (not necessarily isomorphic to $K^n$) to get a construction for $\prsup{a}{\Omega}$. More precisely, given a construction $\scrc \ceq (\kk,\X',\D',\ul{P})$ for $\Omega$ and a finite $\glk$-subgroup $G \subset \aut_L(\kk_L)$, we define $\scrc$ to have symmetry by $G$ if $G$ acts on all three objects $(\kk_L,\X'_L,\D'_L)$ compatibly, and each action is given by a $\glk$-group homomorphism $G \too \aut_L(-)$. This gives rise to a homomorphism of $\glk$-groups $G \too \aut_L(\Omega_L)$ (which in the case $\Omega = K^n$ reduces to the representation $G \too S_n$ in \ref{emp:G-action}), so we may again define the set $\Et(\scrc,L/K) = \{\prsup{a}{\Omega} \mid a \in Z^1(\glk,G)\}$. With this setup, the analogous conclusion to \Cref{prop:twisting-C} follows, i.e. each twist $\prsup{a}{\scrc}$ is a construction for $\prsup{a}{\Omega}$. 
\end{rem}


\section{Twisted forms of the moduli space $M_{g,n}$}\label{sec:moduli}

For this section, we fix positive integers $g$ and $n$ such that $2g+n\geqslant 5$ and $(g,n) \neq (2,1)$. 

\begin{emp}
Let $\calm_{g,n}$ denote the functor which sends a $K$-scheme $S$ to the set of pairs $(C,\ul{P})$, where $C/S$ is a \defi{relative nice curve} of genus $g$ (i.e., $C \too S$ is smooth and projective, with fibers nice curves of genus $g$) and $\ul{P}: (\spec K^n)_S \inj C$ is a closed immersion of $S$-schemes. The immersion $\ul{P}$ is equivalent to the data of $n$ sections $P_i:S \inj C$, and the pair $(C,\ul{P})$ is usually denoted $(C,P_1,\dotsc,P_n)$ in the literature. %

Deligne--Mumford~\cite{deligne-mumford} proved that there exists a quasi-projective $K$-variety $M_{g,n}$, of dimension $3g-3+n$, which is a coarse moduli space for $\calm_{g,n}$. Thus, any pair $(C,\ul{P}) \in \calm_{g,n}(S)$ defines an $S$-valued point $S \too M_{g,n}$, which we denote by $[(C,\ul{P})] \in M_{g,n}(S)$. 

The assumptions on $(g,n)$ at the start of the section imply, by a result of 
Fantechi and Massarenti~\cite{fantechi-massarenti}, that $\aut_K(M_{g,n}) = \aut_{\ol{K}}(M_{g,n})_{\kbar} = S_n$; here, $S_n$ acts naturally on $M_{g,n}$, with $h\in S_n$ sending a point $[(C,\ul{P})] \in M_{g,n}(S)$ to the point $[(C,\ul{P} \circ h)] \in M_{g,n}(S)$. So, for any finite Galois extension $L/K$, we have a bijection
\begin{equation}\label{eq:H1-twist-Mgn}
    H^1(\glk,S_n) \iso \dfrac{\{\textup{$L/K$-twists of $M_{g,n}$}\}}{\textup{$K$-isomorphism}}. 
\end{equation}
We note below that the twists $\prsup{a}{M_{g,n}}$ again admit a modular interpretation. 
\end{emp}


\begin{lem}\label{lem:MgOmega}
Let $a \in \Hom(\glk,S_n)$ be a cocycle and put $\Omega \ceq \prsup{a}{(K^n)}$. Then, the twist
$M_{g,\Omega/K} \ceq \prsup{a}{M_{g,n}}$
is a coarse moduli space for the functor $\calm_{g,\Omega/K}$ which sends a $K$-scheme $S$ to the set of pairs $(C,\ul{Q})$, where $C/S$ is a relative nice curve of genus $g$ and $\ul{Q}:(\spec \Omega)_S \inj C$ is a closed immersion. 
\end{lem}

\begin{proof}
Let $S$ be a $K$-scheme and let $(C,\ul{Q}) \in \calm_{g,\Omega/K}(S)$. Recall from \ref{emp:twisting} that there is an $L$-isomorphism $\phi_a: \spec L^n \iso \spec \Omega_L$ which satisfies $\phi_a \circ a_{\sigma} = \prsup{\sigma}{\phi_a}$. This induces an $S_L$-isomorphism $(\phi_a,\id) : (\spec L^n)_S \iso (\spec \Omega_L)_S$ (the products here are taken over $K$). If we define $\ul{P} : (\spec L^n)_S \inj C_L$ to be the composition $\ul{Q} \circ (\phi_a,\id)$, then we obtain a commutative diagram
\begin{equation*}
    \begin{tikzcd}
        (\spec L^n)_S \arrow[rr, "{(\phi_a,\id)}"] \arrow[dr, "\ul{P}"'] && (\spec \Omega_L)_S \arrow[dl, "\ul{Q}"]\\
        & C_L. 
    \end{tikzcd}
\end{equation*}
Note that $\ul{P}$ is $\glk$-equivariant when $\glk$-acts on $\spec L^n$ by the twisted-by-$a$ action, and on $C_L$ by the natural action. In other words,  $\ul{P} \circ a_{\sigma}\circ \sigma = \sigma \circ \ul{P}$, and hence,  $\prsup{\sigma}{\ul{P}} = \ul{P} \circ a_{\sigma}$. It follows that the morphism $[(C_L,\ul{P})] : S_L \too (M_{g,n})_L$ is $\glk$-equivariant when $\glk$ acts on $(M_{g,n})_L$ via the twisted-by-$a$ action $\sigma \mapsto a_{\sigma} \circ \sigma$. Taking quotients by $\glk$, $[(C_L,\ul{P})]$ descends to the desired morphism $[(C,\ul{Q})]:S \too M_{g,\Omega/K}$. 
\end{proof}


\begin{thm}\label{thm:symmetry-Mgn}
Let $\scrc \ceq (\kk,\X_{\kk},\D_{\kk},\ul{P})$ be a genus $g$ construction for $K^n/K$ with symmetry by $G \subset \aut_K(\kk)$ (as in \Cref{prop:twisting-C}). Let $G$ act on $M_{g,n}$ via the representation $G \too S_n$ from \ref{emp:G-action}. Then, the morphism $[(\X_{\kk},\ul{P})]:\spec \kk \too M_{g,n}$ is $G$-equivariant.  
\end{thm}

\begin{proof}
This follows because, for all $g\in G$, the following diagram commutes (by \Cref{lem:G-action}):
    \begin{equation*}
    \begin{tikzcd}
        (\spec K^n)_\kk \arrow[rr, hook, "{\ul{P}}"]  \arrow[d, "{(\id,g)}"] && \X_{\kk} \arrow[d, "{(\id,g)}"] \arrow[r] & \spec \kk \arrow[d, "g"],\\
        (\spec K^n)_\kk \arrow[rr, hook, "{\ul{P} \circ g}"] && \X_{\kk} \arrow[r] & \spec \kk. 
    \end{tikzcd} \qedhere
    \end{equation*}
\end{proof}

\begin{emp}\label{emp:moduli-point-p}
\textbf{Moduli points.}
Let $p_{\scrc} \in M_{g,n}$ denote the image of $[(\X_{\kk},\ul{P})] \in M_{g,n}(\kk)$; we call this the \defi{moduli point} for $\scrc$. More generally, giving a genus $g$ construction $\scrc \ceq (\ff,\X',\D',\ul{Q})$ for $\Omega/K$ is equivalent to giving a section $(\X',\ul{Q})$ of $\calm_{g,\Omega/K}(\ff)$ (by viewing $\ul{Q}$ as a closed immersion $\ul{Q} : (\spec \Omega)_{\ff} \inj \X'$ with image $\D'$), and we define the moduli point for $\scrc$, denoted $p_{\scrc} \in M_{g,\Omega/K}$, to be the image of  $[(\X',\ul{Q})]\in M_{g,\Omega/K}(\ff)$. 

\Cref{prop:twisting-C} can now be understood in the context of these moduli spaces as follows. Twisting $[(\X_{\kk},\ul{P})]: \spec \kk \too M_{g,n}$ by a cocycle $a \in \Hom(\glk,G)$ gives the morphism  $[(\X_{\ff},\prsup{a}{P})]: \spec \ff \too M_{g,\Omega/K}$ associated to the twist $\prsup{a}{\scrc} \ceq (\ff,\X_{\ff},\D_{\ff},\prsup{a}{\ul{P}})$ (cf. \Cref{prop:twisting-C}). In particular, the $L$-isomorphism $\phi_a: (M_{g,n})_L \iso (M_{g,\Omega/K})_L$ sends the moduli point $p_{\scrc}$ to the moduli point $p_{\prsup{a}{\scrc}}$. 
\end{emp}

\begin{emp}
\textbf{Stabilizers.} In \Cref{thm:symmetry-Mgn}, observe that  the stabilizer $\stab(p_{\scrc})$ contains the image $H$ of $G$ in $S_n$. Thus, constructions $\scrc$ with large symmetry groups give rise to points $p_{\scrc} \in M_{g,n}$ with large stabilizers. It appears to be that case that as $n$ gets larger, the symmetry in constructions for $K^n/K$ gets smaller. Moreover, modifications of existing constructions that increase the number of points often come at the expense of symmetry, as we now illustrate.
\begin{enumerate}[(a)]
    \item The constructions of Nagao~\cite{nagao13} and~\cite{kihara14} modify Mestre's construction~\cite{mestre11} to obtain elliptic curves $E/\Q(t)$ with rank $13$ and $14$, respectively, but these constructions no longer possess the $A_4$-symmetry of~\cite{mestre11}. 
    \item N\'eron's construction outlined in \Cref{exmp:shioda-11} actually begins with a construction of rank $8$ elliptic curves with symmetry by $S_8$, and then modifies this to acquire three new points at the expense of the $S_8$-symmetry.
    \item The construction $\scrc_3$ in \eqref{eq:MS-construction} is a genus $g$ construction ($g\geqslant 2$) for $K^{4g+8}/K$ which can be thought of as a refinement of the Mestre--Shioda Construction (denoted $\scrc_1$ in \eqref{eq:MS-construction}), which is a genus $g$ construction for $K^{4g+6}/K$. However, whereas the latter has symmetry by $S_{4g+6}$, the former has symmetry by the wreath product $\mu_2 \wr S_{2g+4} \subset S_{4g+8}$ (cf. \Cref{sec:det-Et-scrc}).
\end{enumerate}
\end{emp}

\begin{emp}\label{emp:versality}
\textbf{Versality.} 
We call a point $p \in M_{g,n}$ a \defi{unirational point} if its closure $U \subset M_{g,n}$ is a unirational $K$-variety; we call $p$ \defi{very versal} with respect to a subgroup $H \subset S_n$ if $\stab(p)$ contains $H$ and there exists an $H$-equivariant morphism $\spec \ff \too \{p\}$, where $\ff/K$ is a rational function field on which the $H$-action is linear (cf. \ref{emp:linear}). This is equivalent to saying that the $H$-action on $U$ is very versal in the sense of~\cite{versality}*{Page~500}. 

In \Cref{thm:symmetry-Mgn}, if the $G$-action on $\kk$ is linear (cf. \ref{emp:linear}), then the moduli point $p_{\scrc}$ is very versal with respect to $H$ (the image of the associated representation $G \too S_n$). It follows that \emph{for every $a\in \Hom(\glk,G)$, putting $\ff \ceq \prsup{a}{\kk}$ and $\Omega \ceq \prsup{a}{(K^n)}$ as usual, the moduli point  for the twisted construction $\prsup{a}{\scrc}$ is a unirational point of $M_{g,\Omega/K}$.} (Indeed, twisting $[(\X_{\kk},\ul{P})]: \spec \kk \too \{p_{\scrc}\}$ gives the morphism $[(\X_{\ff},\prsup{a}{\ul{P}})]: \spec \ff \too \{p_{\prsup{a}{\scrc}}\}$, and $\ff/K$ is again a rational function field by \Cref{lem:hilbert 90}.)
\end{emp}

\begin{exmp}
For the construction $\scrc$ in \Cref{exmp:joshi}, the $G$-action on $\kk = K(\p_2^9)$ is linear and the moduli point $p_{\scrc}$ is the generic point of $M_{1,9}$, so the above discussion serves to show that \emph{every twisted form of $M_{1,9}/K$ is unirational} (see~\cite{versality}*{Theorem~6.1} for more details). Similarly, for the Mestre--Shioda Construction, denoted $\scrc_1$ in \Cref{sec:mestre-shioda}, the moduli point $p_{\scrc_1} \in M_{g,4g+6}$ is very versal with respect to $S_{4g+6}$. 
\end{exmp}
 
\begin{rem}\label{rem:M-1-n}
Belorousski~\cite{belorousski} proved that $M_{1,n}$ is rational over $\C$ for $n=1,\dotsc,10$ (see~\cite{bini-fontanari}*{Section~1} for a succinct exposition). We are unsure whether this rationality holds over $\Q$, and we also do not know if the twists of $M_{1,10}$ enjoy unirationality properties as in the example above. At any rate, $M_{1,10}$ contains the Mestre--Shioda moduli point $p_{\scrc_1}$ which is very versal with respect to $S_{10}$ (although one can show that $\codim(p_{\scrc_1}) \geqslant 4$).

Bini and Fontanari~\cite{bini-fontanari}*{Theorem~3} compute the Kodaira dimension $\kappa(\ol{M}_{1,n})$ for all $n\geqslant 11$ and find an interesting dichotomy-- they find that $\kappa(\ol{M}_{1,11}) = 0$, and that $\kappa(\ol{M}_{1,n}) = 1$ for all $n\geqslant 12$\footnote{As usual, $\ol{M}_{g,n}$ denotes the moduli space of $n$-pointed stable genus $g$ curves.}. In particular, $M_{1,n}/K$ fails to be unirational for $n\geqslant 11$. In (what appears to be) an amusing numerical coincidence, $11$ is the smallest value of $\deg \Omega$ for which we do not know how to realize $\V{\Omega}{K}$ in an elliptic curve over $K$ (see \Cref{exmp:main} \ref{exmp:main:rohrlich-matsuno} and \Cref{rem:elliptic}). So, we conclude this section with a natural question in this direction. 
\end{rem}

\begin{quest}\label{quest:M1-11}
Does there exist a point $p\in M_{1,11}$ which is very versal with respect to a \emph{transitive} subgroup $H \subset S_{11}$? What if we replace $11$ with some $n\geqslant 12$? 
\end{quest}

\section{The Mestre--Shioda Construction}\label{sec:mestre-shioda}

\begin{emp}\label{emp:setup}
In this section we describe the Mestre--Shioda construction, the starting point of which is \Cref{lem:sqroot} below. For the rest of the article, $K$ denotes a field of characteristic different from $2$, $d$ is a fixed positive integer, and we put
\begin{equation*}
    n \ceq 2d+2.
\end{equation*}
\end{emp}

{ 
\begin{lem}\label{lem:sqroot}
Let $\mbf{m},\mbf{h}$ and $\Bell$ denote the tuples of indeterminates $m_0,\dotsc,m_{n-1}, h_0,\dotsc,h_d$ and $\ell_0,\dotsc,\ell_d$, respectively, and define the polynomials
\begin{equation}
    \begin{alignedat}{3}
        m(x) & \ceq  x^n + m_{n-1}x^{n-1} && + \dotsb + m_0 && \in K[\mbf{m}][x],\\
        h(x) & \ceq  \quad x^{d+1} + h_{d}x^{d} && + \dotsb + h_0 && \in K[\mbf{h}][x],\\
        \ell(x) & \ceq  \qquad\qquad \ell_{d}x^{d} && + \dotsb + \ell_0 && \in K[\Bell][x].
    \end{alignedat}
\end{equation}
Then, the homomorphism $\phi_d: K[\mbf{m}] \too K[\mbf{h},\Bell]$ defined by equating coefficients in the identity
\begin{equation}\label{eq:sqrt-approx}
    m(x) = h(x)^2 - \ell(x)
\end{equation}
is a $K$-algebra isomorphism. 
\end{lem}

\begin{proof}
We construct the inverse $\phi_d\inv$ by by sequentially ``solving'' for the variables $h_{d},\dotsc,h_0,\ell_{d},\dotsc,\ell_0$ (in terms of the $m_i$'s) as follows. We have $h_{d} = m_{n-1}/2$. For $j = 2,\dotsc, d+1$, the coefficient of $x^{n- j}$ in $h(x)^2$ is of the form $2h_{d+1-j} + g_j$ for some polynomial $g_j \in K[h_{d+2-j},\dotsc, h_{d}]$, and the coefficient of $x^{n- j}$ in $-\ell(x)$ is $0$. So, we have $m_{n-j} = 2h_{d+1-j} + g_j$, which yields 
\begin{equation*}
    h_{d+1-j} = (m_{n - j} - g_j)/{2}, \qquad j = 2,\dotsc , d+1.
\end{equation*}
This expresses each coefficient of $h(x)$ as a polynomial in the $m_i$. Next, for $0 \leqslant i \leqslant d$, the coefficient of $x^i$ in  $h(x)^2 - \ell(x)$ is of the form $q_i - \ell_i$, where $q_i \in K[h_0,\dotsc , h_{d}]$, which gives
\begin{equation*}
    \ell_i = g_i - m_i, \quad \quad i=0,\dotsc,d.
\end{equation*}
This expresses each coefficient of $\ell(x)$ as a polynomial in the $m_i$, and concludes the proof that $\phi_d$ is an isomorphism.
\end{proof}}

\begin{emp}\label{emp:curve-div}
We identify the $n$-dimensional rational function field $K(\mbf{h},\Bell)$ with the field
\begin{equation*}
    \mm \ceq K(\mbf{m}) = K(m_0,\dotsc,m_{n-1})
\end{equation*}
using the isomorphism $\phi_d:[\mbf{m}] \iso K[\mbf{h},\Bell]$ from \Cref{lem:sqroot}, and in this way, we view $h(x),\ell(x),$ and $m(x)$ as polynomials in $\mm[x]$ satisfying the identity \eqref{eq:sqrt-approx}. In the table below, we define three hyperelliptic curves $\X_i/\mm\; (i=1,2,3),$ each equipped with an effective divisor $\D_i$. The three polynomials $\ell(x), x\ell(x),$ and $\ell(x^2)$ are separable of degree $d,d+1,$ and $2d$, respectively. Using the genus-degree formula for hyperelliptic curves, we express $d$ in terms of the genera $g_i$ of the curves in the fourth and sixth column. The formulas for $\deg \D_i$ in terms of $g_i$ are proved in \Cref{lem:basic-facts} below.
\begin{table}[H]
\centering
\renewcommand{\arraystretch}{1.35}
\adjustbox{max width=\textwidth}{
\caption{Curves with effective divisors}\label{tab:curve-divisor}
\begin{tabular}{|c|c|c|c|c|c|c|} 
\hline
\multicolumn{3}{|c|}{\phantom{\quad}} & \multicolumn{2}{c|}{$d$ is odd} & \multicolumn{2}{c|}{$d$ is even}\\
\hline
$i$ & Curve $\X_i/\mm$ &  Divisor $\D_i$  & $d$ & $\deg \D_i$ & $d$ & $\deg \D_i$\\ 
\hline
$1$ & $y^2 = \ell(x)$ &  $\divzero{y-h(x)}$ & $2g_1 + 1$ & $4g_1+4$ & $2g_1 + 2$ & $4g_1+6$ \\
\hline
$2$ & $y^2 = x\ell(x)$ &  $\divzero{m(x)}$ & $2g_2+1$ & $8g_2+8$ & $2g_2$ & $8g_2+4$ \\
\hline
$3$ & $y^2 = \ell(x^2)$ & $\divzero{y-h(x^2)}$ & $g_3+1$ & $4g_3+8$ & $g_3+1$ & $4g_3+8$\\
\hline 
\end{tabular}}
\end{table}
Note that these curves fit into a commutative diagram
\begin{equation}\label{diag:morphisms}
    \begin{tikzcd}
    & \X_3 \arrow[dl, "\varphi_1"'] \arrow[d, "{\pi_3}"] \arrow[dr, "\varphi_2"] & & &(x,y) \arrow[dl, mapsto] \arrow[d, mapsto] \arrow[dr, mapsto] &\\
    \X_1 \arrow[dr, "\pi_1"'] & \p^1_{\mm} \arrow[d, "s"] & \X_2 \arrow[dl, "\pi_2"] & (x^2,y) \arrow[dr, mapsto] & x \arrow[d, mapsto] & (x^2,xy) \arrow[dl, mapsto]\\
    & \p^1_{\mm} & & & x^2. &
    \end{tikzcd}
\end{equation}
Next, we define the following two \'etale $\mm$-algebras of degree $n$ and $2n$, respectively\footnote{Saying that $\ee/\mm$ and $\wt{\ee}/\mm$ are \'etale is equivalent to saying that $m(x)$ and $m(x^2)$ are separable.}:
\begin{equation*}\label{eq:ee}
    \ee \ceq \mm[x]/(m(x)), \qquad \wt{\ee} \ceq  \mm[x]/(m(x^2)).
\end{equation*}
\end{emp}

\begin{lem}\label{lem:basic-facts}
The diagram \eqref{diag:morphisms} restricts to a commutative diagram of $\mm$-schemes (with the labels indicating whether the arrows are isomorphisms or finite of degree $2$):
    \begin{equation}\label{diag:E-D}
    \begin{tikzcd}
        & {\D_3} \arrow[dl, "2"'] \arrow[d,"\mathclap{\resizebox{0.30cm}{0.11cm}{$\sim$}}",sloped] \arrow[dr,"\mathclap{\resizebox{0.45cm}{0.1cm}{$\sim$}}",sloped] &\\
        \D_1 \arrow[dr, "\mathclap{\resizebox{0.45cm}{0.1cm}{$\sim$}}"',sloped] & \spec \wt{\ee} \arrow[d, "2"] & \D_2 \arrow[dl, "2"]\\
        & \spec \ee, &
    \end{tikzcd}
    \end{equation}
    In particular, $\D_1,\D_2$ and $\D_3$ are divisors of type $\ee,\wt{\ee},$ and $\wt{\ee}$, respectively. 
\end{lem}


\begin{proof}
Put ${\mathbb{O}} \ceq \mm[x,y]/(y-\ell(x),y-h(x)),$
    so that ${\D_1} = \spec \mathbb{O}$. The identity $m(x) = h(x)^2 - \ell(x)$ from \Cref{lem:sqroot} implies that 
    \begin{equation*}
        m(x) = -(y-h(x))(y+h(x)) = 0 \in \mathbb{O}.
    \end{equation*}
    So, the natural map $\mm[x] \too \mathbb{O}$ factors through ${\ee}$. The homomorphism $\ee \too \mathbb{O}$ so obtained (which corresponds to $\pi_1:\D_1 \too \spec \ee$) admits an inverse given by $x \mapsto x, y \mapsto h(x)$, so $\pi_1:\D_1 \iso \spec \ee$ is an isomorphism. The proof for the isomorphism $\pi_3: {\D_3}\iso \spec \wt{\ee}$ is similar. Visibly, the morphisms $s: \spec \wt{\ee} \too \spec \ee$ and $\pi_2:\D_2 \too \spec \ee$ are \'etale of degree $2$. Since $s\circ \pi_3 = \pi_2\circ \varphi_2$, we conclude that ${\varphi_2}:\D_3 \too \D_2$ is an isomorphism. 
\end{proof}

\begin{emp}\label{emp:split-field-points}
\textbf{Splitting fields and points.} Let $\kk$ and $\wt{\kk}$ denote the rational function fields $K(u_1,\dotsc,u_{n})$ and $K(t_1,\dotsc,t_{n})$, respectively. We define field embeddings\footnote{Here, $s_{n-j}(\mbf{u})$ denotes the $(n-j)$-th elementary symmetric polynomial in $u_1,\dotsc,u_{n}$.}
\begin{equation*}
    \begin{alignedat}{3}
    \mm & \inj \kk, \qquad m_{j} && \mtoo s_{n-j}(\mbf{u}), \qquad && \textup{ for }  j=0,\dotsc,n-1,\\
    \kk & \inj \wt{\kk}, \qquad\;   u_i && \mtoo t_i^2, \qquad && \textup{ for } i=1,\dotsc,n.
    \end{alignedat}
\end{equation*}
By construction, the embeddings $\mm \inj \kk \inj \wt{\kk}$ give rise to the factorizations 
\begin{equation}\label{eq:factorization}
    \begin{alignedat}{2}
        m(x^2) &= (x-t_1)(x+t_1)\dotsb (x-t_{r})(x+t_{n}) &&\; \in \wt{\kk}[x],\\
        m(x) &= (x-u_1)\dotsb(x-u_{n}) &&\; \in \kk[x].
    \end{alignedat}
\end{equation}
By \Cref{lem:basic-facts}, $\kk$ (resp. $\wt{\kk}$) is a splitting field for $\D_1$ (resp. $\D_2$ and ${\D_3}$). If $\iota \in \aut \X_2$ denotes the involution $(x,y) \mapsto (x,-y)$ and $\tau \in \aut \X_3$ the involution $(x,y) \mapsto (-x,y)$, then
\begin{equation*}\label{eq:div-points}
    \begin{alignedat}{3}
    \D_1(\kk) &= \{P_1,\dotsc,P_{n}\}, \quad && \textup{where} \quad P_i &&\ceq (u_i,h(u_i)),\\
    \D_2(\wt{\kk}) &=\{Q_1, \iota(Q_1),\dotsc,Q_{n},\iota(Q_{n})\}, \quad && \textup{where} \quad Q_i &&\ceq (u_i,t_ih(u_i)),\\
    {\D_3}(\wt{\kk}) &=\{R_1,\tau(R_1),\dotsc,R_{n},\tau(R_{n})\}, \quad && \textup{where} \quad R_{i} &&\ceq (t_i,h(u_i)).
    \end{alignedat}
\end{equation*}
These points determine isomorphisms $\ul{P},\ul{Q},\ul{R}$ as in \ref{emp:constructions}, and we see that
    \begin{equation}\label{eq:MS-construction}
        \begin{alignedat}{5}
        \scrc_{1} &\ceq (\kk,(\X_1)_{\kk},(\D_1)_{\kk},\ul{P}) &&\; \textup{ is a genus } && g_1 && \textup{ construction for } && K^{n}/K,\\
        \scrc_{2} &\ceq (\wt{\kk},(\X_2)_{\wt{\kk}},(\D_2)_{\wt{\kk}},\ul{Q}) &&\; \textup{ is a genus } && g_2 &&\textup{ construction for } && K^{2n}/K,\\
        {\scrc_3} &\ceq (\wt{\kk},(\X_3)_{\wt{\kk}},(\D_3)_{\wt{\kk}},\ul{R}) &&\; \textup{ is a genus } && g_3 && \textup{ construction for } && K^{2n}/K.
        \end{alignedat}
    \end{equation}
The construction $\scrc_1$ was described in \ref{emp:MSC}.  We think of the constructions $\scrc_2$ and ${\scrc_3}$ as refinements of $\scrc_1$; for example, $\scrc_1$ can be specialized to yield genus $g$ curves $X/K$ with $\#X(K)\geqslant 8g+12$, whereas ${\scrc_3}$ yields genus $g$ curves $X/K$ with $\#X(K)\geqslant 8g+16$. 

Note that if $d=4$ then $\scrc_1$ is a genus one construction for $K^{10}/K$ (i.e., $\X_{\kk}/\kk: y^2 = \ell(x)$ is a genus one curve with $10$ $\kk$-rational points). Similarly, if $d=3$ then $\scrc_2$ is a genus one construction for $K^{16}/K$.
\end{emp}


\section{The dimensions of $\W{\D_1},\W{\D_2},$ and $\W{\D_3}$}

\begin{emp}\label{emp:morphism-jac}
For $i=1,2,3$, let $\J_i/\mm$ denote the Jacobian of $\X_i/\mm$. The morphisms $\varphi_1$ and $\varphi_2$ (see \eqref{diag:morphisms}) give rise to an $\mm$-isogeny $((\varphi_1)_*,(\varphi_2)_*): \J_3 \too \J_1\times \J_2$ (see~\cite{shiodasymmetry}*{Proposition~3}), which induces an isomorphism of $\qgm$-modules
    \begin{equation*}
        ((\varphi_1)_*,(\varphi_2)_*):\J_3(\ol{\mm})_{\Q} \iso \J_2(\ol{\mm})_{\Q} \oplus \J_1(\ol{\mm})_{\Q}.
    \end{equation*}
Since $\varphi_1$ and $\varphi_2$ give surjections ${\D_3}(\ol{\mm}) \surj \D_1(\ol{\mm})$ and ${\D_3}(\ol{\mm}) \surj \D_2(\ol{\mm})$, respectively (by \Cref{lem:basic-facts}), the isomorphism above restricts to an isomorphism of $\qgm$-modules
\begin{equation}\label{eq:VDM-isom}
    \W{\D_3} \iso \W{\D_2} \oplus \W{\D_1}.
\end{equation}
\end{emp}

\begin{thm}[Shioda, 1998]\label{thm:shioda-dim}
We have
\begin{align*}\label{eq:dim-VDM}
    \dim_{\Q} \W{\D_1} &= n-1,\\
    \dim_{\Q} \W{\D_2} &= n,\\
    \dim_{\Q}\W{\D_3} &= 2n-1.\\ 
\end{align*}
\end{thm}

\begin{proof}
The statement for $\W{\D_1}$ and $\W{\D_2}$ is a reformulation of~\cite{shiodasymmetry}*{Theorems~5, 6} and~\cite{shiodasymmetry}*{Theorem~7}, respectively. The statement for $\W{\D_3}$ then follows from \eqref{eq:VDM-isom}. 
\end{proof}

\section{Zarhin's criterion for simplicity}\label{sec:zarhin}

To prove the simplicity statements in Theorems \ref{thm:main} and \ref{thm:quad}, we will apply Hilbert Irreducibility to \Cref{prop:gal=Sd-1} below and then appeal to the following result of Zarhin.

\begin{thm}[{\cite{zarhinSimple}*{Theorems~1.1, 1.3}}]\label{thm:zarhin-simple}
Suppose $f(x) \in K[x]$ is a degree $d$ polynomial with Galois group $S_d$, and $X/K$ is a hyperelliptic curve of the form:
\begin{enumerate}[(a),itemindent=0pt,leftmargin=1cm]
    \item $X/K: y^2 = f(x)$, or \label{item:zarhin:a}
    \item $X/K: y^2 =  (x-a)f(x)$, with $d$ odd, $a \in K$, and $f(a) \neq 0$.\label{item:zarhin:b}
\end{enumerate}
Assume that $d\geqslant 5$, and if $\chr K =3$, that $d\geqslant 7$. Then, $\End_{\kbar}((J_X)_{\kbar})=\Z$; in particular, $J_X/K$ is absolutely simple.
\end{thm}

\begin{prop}\label{prop:gal=Sd-1}
The polynomial $\ell(x) \in \wt{\kk}[x]$ (cf. \ref{emp:curve-div}) has Galois group $S_{d}$. 
\end{prop}

\begin{proof}
Recall from \ref{emp:setup} that we identify the rational function field $K(\mbf{h},\Bell)$ with $\mm$ using the isomorphism $\phi_d:K[\mbf{m}] \iso K[\mbf{h},\Bell]$ from \Cref{lem:sqroot}. So, we have a chain of field extensions
\begin{equation*}
    K(\Bell) \inj K(\Bell,\mbf{h}) = \mm \inj \wt{\kk},
\end{equation*}
in which the first extension is purely transcendental of degree $d+1$, and the second is finite Galois. Let $L/K(\Bell)$ denote the splitting field of $\ell(x)$, and put $\L \ceq L\otimes_{K(\Bell)} \wt{\kk}$. The polynomial $\ell(x)\in K(\Bell)[x]$ has Galois group $S_{d}$, i.e., $\gal(L/K(\Bell)) \cong S_{d}$. It suffices therefore to show that $\L$ is a field, because the canonical group embedding  $(-)\otimes \id : \gal(L/K(\Bell)) \inj \aut_{\wt{\kk}}(\L)$ would then necessarily be an isomorphism (since $[\L:\wt{\kk}]=(d-1)! = \#S_{d}$). The only way $\L$ can fail to be a field is if there is some intermediate field $K(\Bell) \subsetneq F \subset L$ which is contained in $\wt{\kk}$. Thus, it suffices to show that \textit{$K(\Bell)$ is algebraically closed in $\wt{\kk}$.} 

Assume for a contradiction, then, that we have a finite extension $F \supsetneq K(\Bell)$ contained in $\wt{\kk}$. Let $R$ be the integral closure of $K[\Bell]$ in $F$, and put $S \ceq R\otimes_{K[\Bell]} K[\mbf{m}]$. Then, $S$ is an integral extension of $K[\mbf{m}]$ contained in $\wt{\kk}$, and hence, $S$ is a sub-ring of $K[\mbf{t}]$, which is the integral closure of $K[\mbf{m}]$ in $\wt{\kk}$. We have therefore a commutative diagram
\begin{equation*}
    \begin{tikzcd}
        \A(\mbf{t}) \arrow[r] & \spec S \arrow[r]\arrow[d] & \A(\mbf{m})\arrow[d, "f"]\\
        & \spec R \arrow[r] & \A(\Bell),
    \end{tikzcd}
\end{equation*}
in which the horizontal arrows are finite morphisms.  Now, let $B$ denote the branch locus (i.e., the set of branch points) of the finite morphism $\spec R \too \A(\Bell)$, so that $f\inv (B)$ is the branch locus of $\spec S \too \A(\mbf{m})$. The point $(0,\dotsc,0) \in \A(\mbf{t})(\kbar)$ is totally ramified under the morphism $\A(\mbf{t}) \too \A(\mbf{m})$, which implies that $f\inv (B)$ is non-empty. So, $f\inv (B)$ is the vanishing locus (in $\A(\mbf{m})$) of some non-constant, square-free element $\delta \in K[\Bell]$.  The branch locus $\wt{B}$ of $\A(\mbf{t}) \too \A(\mbf{m})$ is the vanishing locus of $\Delta \ceq \disc \, m(x^2) \in K[\mbf{m}]$. The containment $f\inv(B) \subset \wt{B}$ implies that $\delta$ divides $\Delta \in K[\mbf{m}]$. But this is a contradiction because the only irreducible factors of $\Delta$ in $K[\mbf{m}]$ are $\disc \,m(x)$ (see ~\cite{odoni}*{Lemma~8.1}) and $m_0 = h_0^2 - \ell_0$, both of which are not contained in the subring $K[\Bell] \subsetneq K[\mbf{m}]$.
\end{proof}

\begin{rem}
In the setting of N\'eron's Theorem (\Cref{thm:neron}), Noot~\cite{noot}*{Corollary~1.5} showed using an argument of Serre~\cite{noot}*{Proposition~1.3} that if $K$ is finitely generated of characteristic $0$, then the set of points $P\in U(K)$ for which $\End_{\kbar}((\calj_P)_{\kbar})$ is larger than $\End_{\ol{\ff}}(\J_{\ol{\ff}})$ is thin. Although he states this for a single specialization, the proof of~\cite{noot}*{Proposition~1.3} yields the above stronger conclusion; see for example the discussion at the end of~\cite{masser-endo}*{Page~461}. 
\end{rem}

\section{Determining the sets $\Et(\scrc_i,L/K)$}\label{sec:det-Et-scrc}

\begin{emp}\label{emp:E}
Fix for this section a finite Galois extension $L/K$. The construction $\scrc_1 \ceq (\kk,(\X_1)_{\kk},(\D_1)_{\kk},\ul{P})$ has symmetry by $G \ceq G_{\kk/\mm}$ (cf. \Cref{exmp:easy-symmetry}), which is isomorphic to the symmetric group $S_n$ acting naturally by permutation on $u_1,\dotsc,u_n$ (see the proof of \Cref{prop:wreath-gal-isom} below). Thus, the associated representation $G \too S_n$ is an isomorphism, which immediately implies that $\Et(\scrc_1,L/K) = \Et(n,L/K)$.

On the other hand, the constructions $\scrc_2$ and $\scrc_2$ both admit symmetry by $\wt{G} \ceq \gg{\wt{\kk}/\mm}$ (cf. \Cref{exmp:easy-symmetry}). We determine the precise structure of this group in \Cref{prop:wreath-gal-isom} below, and we determine the sets $\Et(\scrc_i,L/K)$, for $i=2,3$, in \Cref{prop:symmetry-scrc}.
\end{emp}
We determine the sets $\Et(\scrc_i,L/K)$, for $i=2,3$, in \Cref{prop:symmetry-scrc} below. 

\begin{emp}\label{emp:wreath-product}
\textbf{The group $\mu_2 \wr S_n$.} 
Let $\mu_2 =\{1,-1\}$ be the cyclic group of order $2$, and view $\mu_2^n$ as the group of set-maps $\{1,\dotsc,n\} \too \mu_2$. For an element $h \in \mu_2^n$ and $i\in\{1,\dotsc,n\}$, we put $h_i \ceq h(i)$. The symmetric group $S_n$ acts naturally on $\mu_2^n$ by 
\begin{equation*}
    \prsup{g}{h} \ceq h\circ g\inv, \; \textup{ for } g\in S_n \textup{ and } h\in \mu_2^n.
\end{equation*}
The associated semi-direct product $\mu_2^n \rtimes S_n$, denoted $\mu_2 \wr S_n$, is called the \defi{wreath product} of $\mu_2$ by $S_n$ (see~\cite{odoni}*{Section~4} for background on wreath products arising as Galois groups of composite polynomials). 
\end{emp}

\begin{prop}\label{prop:wreath-gal-isom}
We have a commutative diagram of groups
\begin{equation}\label{diag:wreath-gal-isom}
    \begin{tikzcd}
    \mu_2 \wr S_n \arrow[d, "{\resizebox{0.45cm}{0.1cm}{$\sim$}}" labl1] \arrow[r,twoheadrightarrow] & S_n \arrow[d, "{\resizebox{0.45cm}{0.1cm}{$\sim$}}" labl1]  & (h,g) \arrow[d, mapsto] \arrow[r, mapsto] & g \arrow[d, mapsto]\\
    \wt{G} \arrow[r,twoheadrightarrow] & \gg{\kk/\mm} & (t_i \mapsto h_{g(i)} t_{g(i)}) \arrow[r, mapsto] & (u_i \mapsto u_{g(i)}).
    \end{tikzcd}
\end{equation}
\end{prop}

\begin{proof}
The map $S_n \too \gg{\kk/\mm}$ in \eqref{diag:wreath-gal-isom} is an isomorphism by the Fundamental Theorem of Symmetric Functions. The map $\mu_2 \wr S_n \too \wt{G}$ in \eqref{diag:wreath-gal-isom}, call it $f$, restricts to an isomorphism from $\mu_2^n \times \{\id\}$ to $\gg{\wt{\kk}/\kk}$. Thus, $f$ is a bijection of sets. One checks easily that it is a homomorphism, so we conclude that it is an isomorphism. 
\end{proof}

\begin{emp}\label{emp:phi-isom}
Let $\mf{R}$ and $\wt{\mf{R}}$ denote the set of roots of $m(x)$ and $m(x^2)$, respectively, i.e.
\begin{equation*}
\begin{alignedat}{2}
    \mf{R} & \ceq \{u_1,\dotsc,u_{r}\} && \subset \kk,\\
    \wt{\mf{R}} & \ceq \{t_1,-t_1,\dotsc,t_{r},-t_{r}\} && \subset \wt{\kk}.
\end{alignedat}
\end{equation*}
Write $\{e_{r}\}_{r\in \mf{R}}$ and $\{\wt{e}_{r}\}_{r\in \wt{\mf{R}}}$ for the standard basis vectors of $K^{r}$ and $K^{2r}$, respectively.
We have a natural double cover $\spec K^{2n} \too \spec K^n$ defined on rings by
\begin{equation}\label{eq:K2d-inj-K4d}
    K^{n} \inj K^{2n}, \qquad \sum_{i=1}^{n} b_i e_{u_i} \mtoo \sum_{i=1}^{n} b_i(\wt{e}_{t_i} + \wt{e}_{-t_i}).
\end{equation}
\end{emp}

\begin{lem}\label{lem:wt-G-subgroup}
The subgroup $\wt{G} \subset \aut_{K}(K^{2n})$ consists exactly of the automorphisms which restrict to an automorphism of $K^n$. 
\end{lem}

\begin{proof}
Let $P \subset \aut_K(K^{2n})$ denote the subgroup of automorphisms which restrict to an automorphism of $L^n$; visibly, we have $\wt{G} \subset P$, so it suffices to show that $\#P = 2^{n}\cdot n!$. We have a natural surjection $P \surj S_n$, whose kernel is the group $N \ceq \aut_{K^n}(K^{2n})$, so it suffices to show that $\#N = 2^{n}$. This follows by noting that $N$ is isomorphic to the deck transformation group of the two-to-one cover $\sqcup_{i=1}^{2n} (\spec K) \surj \sqcup_{i=1}^{n} (\spec K)$.
\end{proof}

If we let $\wt{G}$ act on $K^n$ via the projection $\wt{G} \surj S_n$ (that is, via the $\wt{G}$-action on $\mf{R}$), then $K^n\inj K^{2n}$ is  $\wt{G}$-equivariant. So, twisting this inclusion by a cocycle $a \in \Hom(\glk,\wt{G})$ gives an inclusion $\prsup{a}{(K^n)} \inj \prsup{a}{(K^{2n})}$. 

Recall from \Cref{lem:basic-facts} that both $\D_2$ and $\D_3$ are isomorphic to $\spec \wt{\ee} = \spec \wt{\kk}[x]/m(x^2)$. Thus, the sets $\Et(\scrc_2,L/K)$ and $\Et(\scrc_3,L/K)$ are both determined by the $\wt{G}$-action on $\mf{R}$ (cf. \ref{emp:Et-C}) and they are equal; call this set $S$. 
Observe that if $\wt{\Omega} \in S$ (that is, if $\wt{\Omega} = \prsup{a}{(K^{2n})}$ for some $a \in \Hom(\glk,\wt{G})$), then $\wt{\Omega}$ is \'etale of degree $2$ over the $K$-subalgebra $\Omega \ceq \prsup{a}{(K^n)}$.  We observe that $S$ consists precisely of all such $\wt{\Omega} \in \Et(2n,L/K)$.  



{ 
\begin{prop}\label{prop:symmetry-scrc}
The set $\Et(\scrc_2,L/K)=\Et(\scrc_3,L/K)$ consists of all $\wt{\Omega} \in \Et(2n,L/K)$ such that $\wt{\Omega}$ is \'etale of degree $2$ over a subalgebra $\Omega \subset \wt{\Omega}$. Moreover, given such a pair $(\wt{\Omega},\Omega)$, there exists a cocycle $a\in \Hom(\glk,\wt{G})$ such that $(\wt{\Omega},\Omega) = (\prsup{a}{(K^{2n})},\prsup{a}{(K^{n})})$.
\end{prop}
}
\begin{proof}
Suppose we are given a pair $(\wt{\Omega},\Omega)$ as above. Choose $L$-isomorphisms $p:\Omega_{L} \iso L^n$ and $\wt{p}:\wt{\Omega}_{L} \iso L^{2n}$ which fit into a commutative diagram of $L$-algebras
\begin{equation}
    \begin{tikzcd}
    \wt{\Omega}_{L} \arrow[r, "\wt{p}", "{\resizebox{0.45cm}{0.1cm}{$\sim$}}"'] & L^{2n}\\
    \Omega_{L} \arrow[u, hook] \arrow[r, "{p}", "{\resizebox{0.45cm}{0.1cm}{$\sim$}}"'] & L^n \arrow[u, hook].
    \end{tikzcd}
\end{equation}
Let $a \in \Hom(\glk,S_{2n})$ be the difference cocycle $\sigma \mtoo \wt{p} \circ (\id \otimes\sigma) \circ \wt{p}\inv \circ (\id \otimes \sigma) \inv$ associated to $(\wt{\Omega},\wt{p})$ (see \ref{emp:twisting}). 
For any $\sigma \in \glk$, the automorphism $a_{\sigma} \in S_{2n} = \aut_L(L^{2n})$ restricts to an automorphism of $L^n$. \Cref{lem:wt-G-subgroup} implies that the image of $a$ lands in the subgroup $\wt{G} \subset S_{2n}$, so we may view $a$ as a cocycle in $\Hom(\glk,\wt{G})$ and conclude that $\wt{\Omega} = \prsup{a}{(K^{2n})}$ and $\Omega = \prsup{a}{(K^{n})}$.
\end{proof}

\section{Twisting the Mestre--Shioda construction}\label{sec:twisting-MS}

Now, we twist the Mestre--Shioda construction(s) from \eqref{eq:MS-construction}. We continue with the fixed notation of the latter section, i.e., $d$ is a fixed positive integer and $n = 2d+2$. 

\begin{prop}\label{prop:twist-MS}
Let $\Omega$ be a \fetk of degree $n$, and $\wt{\Omega}$ a finite \'etale $\Omega$-algebra of degree $2$. Then, there exist finite field extensions $\wt{\ff}/\ff/\mm$ with the following properties.
\begin{enumerate}[(a),itemindent=0pt,leftmargin=1cm]
    \item \label{prop:twist-MS:rational}$\wt{\ff}$ and $\ff$ are $n$-dimensional rational function fields over $K$.
    \item \label{prop:twist-MS:galSd-1}The polynomial $\ell(x) \in \mm[x]$ has Galois group $S_{d}$ over $\ff$ and $\wt{\ff}$. 
    \item \label{prop:twist-MS:galoismodule}$\W{(\D_1)_{\ff}},\W{(\D_2)_{\wt{\ff}}},$ and $\W{(\D_3)_{\wt{\ff}}}$ are $\qgk$-modules, and we have isomorphisms
    \begin{align*}
        \W{(\D_1)_{\ff}} &\cong \V{\Omega}{K},\\
        \W{(\D_2)_{\wt{\ff}}} &\cong \V{\wt{\Omega}}{\Omega},\\
        \W{(\D_3)_{\wt{\ff}}} &\cong \V{\wt{\Omega}}{K}. 
    \end{align*}
\end{enumerate}
\end{prop}

\begin{proof}
Let $L/K$ be a finite Galois extension which splits $\wt{\Omega}/K$. By \Cref{prop:symmetry-scrc}, we can choose a cocycle $a\in \Hom(\glk,\wt{G})$ such that $(\wt{\Omega},\Omega) = (\prsup{a}{(K^{2n})},\prsup{a}{(K^{n})})$. The group $\wt{G}$ acts on the field $\wt{\kk}_L$, and on its subfield $\kk_L$ (via the projection to $S_n$). The inclusion $\kk_L \hookrightarrow \wt{\kk}_L$ is $\wt{G}$-equivariant, so, letting $\ff$ and $\wt{\ff}$ denote the twists $\prsup{a}{\kk}$ and $\prsup{a}{\wt{\kk}}$, respectively, we have an inclusion of fields $\ff \hookrightarrow \wt{\ff}$. We check now that the various parts of the proposition hold.
\begin{enumerate}[(a)]
    \item The $\wt{G}$ action on $\kk_L$ and $\wt{\kk}_L$, defined in \eqref{diag:wreath-gal-isom}, is visibly linear (cf. \ref{emp:linear}), so \Cref{lem:hilbert 90} implies that $\ff/K$ and $\wt{\ff}/K$ are $n$-dimensional rational function fields.
    \item \Cref{prop:gal=Sd-1} says that $\ell(x)$ has Galois group $S_{d}$ over $\wt{\kk}_L$. \Cref{prop:twisting-C} \ref{prop:twisting-C-isom-mm} says that $\wt{\ff}_L \cong \wt{\kk}_L$ as $\mm_L$-algebras, and hence, as $L(\Bell)$-algebras. It follows that $\ell(x)$ has Galois group $S_{d}$ over $\wt{\ff}_L$, and hence, also over $\wt{\ff}$. 
    \item \Cref{prop:twisting-C} \ref{prop:twisting-C-isom-twist} says that twisting $\ul{Q}$ and $\ul{R}$ by the cocycle $a$ gives $\wt{\ff}$-isomorphisms
    \begin{equation*}
    	\prsup{a}{\ul{Q}} : (\spec \wt{\Omega})_{\ff} \iso (\D_2)_{\wt{\ff}}, \qquad \prsup{a}{\ul{R}} : (\spec \wt{\Omega})_{\ff} \iso (\D_3)_{\wt{\ff}}.
    \end{equation*}
    Note that $\ff=\prsup{b}{\kk}$, where $b\in \Hom(\glk,S_n)$ is the composition of $a$ with the projection $\wt{G} \surj S_n$.  So, \Cref{prop:twisting-C} \ref{prop:twisting-C-isom-twist} says that we can twist $\ul{P}$ by $b$ to get an $\ff$-isomorphism  $$\prsup{b}{\ul{P}}: (\spec \Omega)_{\ff} \iso (\D_1)_{\ff}.$$
    
    Now, for $i=1$ and $3$, we have $\dim_{\Q} W(\D_i) = \deg \D_i - 1$ by \Cref{thm:shioda-dim}, so the first and third isomorphisms follow from \Cref{prop:twisting-C} \ref{prop:twisting-C-isom-dim}. Since $(\D_1)_{\wt{\ff}} \cong \Omega_{\wt{\ff}}$, we again have that $\W{(\D_1)_{\wt{\ff}}} \cong \V{\Omega}{K}$. By \eqref{eq:VDM-isom}, we have a $\Q \gg{\wt{\ff}}$-module isomorphism $\W{(\D_3)_{\wt{\ff}}} \cong \W{(\D_2)_{\wt{\ff}}} \oplus \W{(\D_1)_{\wt{\ff}}}.$ It follows from the definition of $\V{\wt{\Omega}}{\Omega}$ in \ref{emp:V-wt-Omega} that $\W{(\D_2)_{\wt{\ff}}} \cong \V{\wt{\Omega}}{\Omega}$. \qedhere
\end{enumerate}
\end{proof}

\begin{emp}\label{emp:coordinatize}
\textbf{Coordinatizing the field $\ff$.} 
The construction $\prsup{b}{\scrc_1} \ceq (\ff,(\X_1)_{\ff},(\D_1)_{\ff},\prsup{b}{\ul{P}})$ appearing in the proof of \Cref{prop:twist-MS} above is the {Liu--Lorenzini Construction} described in \ref{emp:LLC}. We explain now how to obtain algebraically independent elements $z_1,\dotsc,z_{n} \in \ff$ such that $\ff = K(\mbf{z})$. With this, we fully recover \ref{emp:LLC} as a twist of \ref{emp:MSC}, as claimed in \ref{emp:outline-layout}. 

Identify $\D_1$ with $\spec \ee$, where $\ee = \mm[x]/m(x)$, so that $\prsup{b}{\ul{P}}$ is given on coordinate rings by an $\ff$-algebra isomorphism
\begin{equation*}\label{eq:twist-theta}
    \prsup{b}{\ul{P}^{\#}} :  \ee_{\ff} = \ff[x]/m(x) \iso \Omega_{\ff}.
\end{equation*}
Put $\alpha \ceq \prsup{b}{\ul{P}^{\#}}(x)\in \Omega_{\ff}$. Then, a choice of $K$-basis $\{\alpha_1,\dotsc,\alpha_{n}\}$ for $\Omega$ determines a tuple $\mbf{z}\ceq z_1,\dotsc,z_{n} \in \ff$ such that $\alpha = \alpha_1 \otimes z_1 + \dotsb + \alpha_{n}\otimes z_{n} \in \Omega_\ff.$
Let $\ff'$ be the subfield of $\ff$ generated by $K$ and the tuple of elements $\mbf{z}$. We claim now that $\ff = \ff'$. 
To see this, write $\chi(x) \ceq x^{n} + f_{n-1}(\mbf{z})x^{n-1} + \dotsb + f_0(\mbf{z}) \in \ff'[x]$ for the characteristic polynomial of $\alpha \in \Omega_{\ff}$. Since $m(x) = 0 \in \ff[x]/m(x)$, it follows that $m(\alpha) = 0 \in \Omega_{\ff}$, so we must have $m(x) = \chi(x) \in \ff'[x].$ Since the coefficients of $m(x)$ generate $\mm$ over $K$, we have inclusions $\mm \subset \ff' \subset \ff$. Writing $\Hom_K(\Omega,L) = \{\varepsilon_1,\dotsc,\varepsilon_1\}$, we have the identity
\begin{equation*}
	m(x) = \prod_{i=1}^{n} \big(x - (z_1 \otimes \varepsilon_i(\alpha_1) + \dotsb + z_{n}\otimes \varepsilon_i(\alpha_{n})) \big) \in \ff'_L[x].
\end{equation*}
It follows that the map
\begin{equation*}
    u_i \mtoo z_1 \otimes \varepsilon_i(\alpha_1) + \dotsb + z_{n}\otimes \varepsilon_i(\alpha_{n}) \in \ff'_L, \qquad i=1,\dotsc,n,
\end{equation*}
defines an $\mm_L$-algebra isomorphism $\kk_L \iso \ff'_L$, which immediately implies that $\ff = \ff'$. Since $\trdeg_K \ff =n$, we conclude that the $z_i$'s are algebraically independent over $K$ and $\ff = K(\mbf{z})$. We note also, after the fact, that the inclusion $\mm \inj \ff$ coincides with the inclusion in \ref{emp:LLC}. 
\end{emp}

\section{Tying it all together- proofs of the main theorems}

\begin{emp}\label{proof-main}
\textit{Proof of Theorem \ref{thm:main}.} 
Assume the hypotheses of \Cref{thm:main}, i.e., $K$ is a Hilbertian field of characteristic different from $2$ and $\Omega$ is a \fetk of degree $n\geqslant 1$. Let $g$ be a positive integer which satisfies $4g+6 \geqslant n$. Our goal is to produce infinitely many genus $g$ hyperelliptic curves $X/K$ such that \ref{thm:main:realize} $J_X/K$ realizes $\V{\Omega}{K}$, and \ref{thm:main:simple} $J_X/K$ is absolutely simple if $g\neq 2$ or $\chr K \neq 3$. In proving \ref{thm:main:realize}, it suffices to replace $\Omega$ with $\Omega' \ceq \Omega \times K^{4g+6-n}$ since $\V{\Omega}{K} \subset \V{\Omega\times K^{4g+6-n}}{K}$. Thus, we assume without loss of generality that $n=4g+6$. Putting $d\ceq 2g+2$ in \Cref{sec:mestre-shioda}, we find that $\scrc_1$ is a genus $g$ construction for $K^n/K$. \Cref{prop:twist-MS} gives a rational function field $\ff \ceq \prsup{a}{\kk}$ such that the twist $\prsup{a}{\scrc_1}$ is a genus $g$ construction for $\Omega/K$. Putting $U\ceq \spec R$, where $R$ is the integral closure in $\ff$ of $K[\mbf{m},(\disc\,\ell(x))\inv]$, the curve $(\X_1)_{\ff}/\ff$ extends to a smooth, projective relative curve $\calx_1/U$ containing the open affine $\spec R[x,y]/(y^2 - \ell(x))$. Let $\calj_1/U$ denote the relative Jacobian of $\calx_1/U$ (see~\cite{neronmodels}*{Proposition~9.4.4}). Now, we specialize:
\begin{enumerate}[(i)]
    \item By \Cref{prop:twist-MS} \ref{prop:twist-MS:galoismodule}, $(\J_1)_{\ff}/\ff$ realizes $\V{\Omega}{K}$, so \Cref{lem:specialization} gives a thin set $T \subset U(K)$ such that for $P \in U(K) \backslash T$, the Jacobian  of $(\calx_1)_P/K$ realizes $\V{\Omega}{K}$.
    \item By \Cref{prop:twist-MS} \ref{prop:twist-MS:galSd-1}, $\ell(x) \in \ff[x]$ has Galois group $S_{d}$, so the classical Hilbert Irreducibility Theorem~\cite{serremw}*{Section~11.1} gives a thin set $T' \subset U(K)$ such that for $P \in \calb(K) \backslash T'$, the specialized polynomial $\ell_P(x) \in K[x]$ has Galois group $S_{d}$. Since $\deg \ell_P(x) = d-1$, Zarhin's criterion (\Cref{thm:zarhin-simple} \ref{item:zarhin:a}) implies that the Jacobian of the curve $(\calx_1)_P/K:y^2 = \ell_P(x)$ is absolutely simple provided $g\neq 2$ (i.e., $\deg\, \ell(x) \neq 6$) or $\chr K \neq 3$.  
\end{enumerate}
Put $S\ceq U(K) \backslash (T\cup T')$.
Since $K$ is assumed to be Hilbertian, $S$ is Zariski-dense in $U$, and since the family $\calx_1/U$ varies in moduli (i.e., the classifying morphism $V \too M_{g,K}$ is non-constant), the set $\{(\calx_1)_P/K\}_{P \in S}$ contains infinitely many genus $g$ hyperelliptic curves. Each of these curves witnesses parts \ref{thm:main:realize} and \ref{thm:main:simple} of the theorem. 

Finally, if $n\leqslant 10$, then the curves $\{ (\calx_1)_P/K\}_{P \in S}$ can be chosen to be of genus one. The set $\{ (\calj_1)_P/K\}_{P \in S}$ then contains infinitely many elliptic curves having pair-wise distinct $j$-invariants and realizing the Galois module $\V{\Omega}{K}$. This concludes the proof. \qed
\end{emp}

\begin{emp}\label{proof-quad}
\textit{Proof of \Cref{thm:quad}.} Assume the hypotheses of \Cref{thm:quad}, i.e., $K$ is a Hilbertian field of characteristic different from $2$, $\Omega$ is a \fetk of degree $n\geqslant 1$, and $\wt{\Omega}$ is an \'etale $\Omega$-algebra of degree $2$. Let $g$ be a positive integer which satisfies $4g+4\geqslant n$. Our goal is to produce infinitely many genus $g$ hyperelliptic curves $X/K$ such that \ref{thm:quad:realize} $J_X/K$ realizes $\V{\wt{\Omega}}{\Omega}$, and \ref{thm:quad:simple} $J_X/K$ is absolutely simple if $g\neq 2$ or $\chr K \neq 3$. To prove \ref{thm:quad:realize}, it suffices to replace $\Omega$ (resp. $\wt{\Omega}$) with $\Omega \times K^{m}$ (resp. $\wt{\Omega} \times K^{2m}$), where $m\ceq 4g+4-n$; this follows because $\V{\wt{\Omega}}{\Omega} \subset \V{(\wt{\Omega} \times K^{2m})}{(\Omega \times K^m)}$ (cf. \Cref{lem:V-wt-Omega-identity}). So, we assume without loss of generality that $n=4g+4$. Then, setting $d=2g+2$ in \Cref{sec:mestre-shioda}, we see that $\scrc_2$ is a genus $g$ construction for $K^{2n}/K$. \Cref{prop:twist-MS} gives a $(2d+2)$-dimensional rational function field $\wt{\ff} \ceq \prsup{a}{\wt{\kk}}$ such that the Jacobian $(\J_2)_{\wt{\ff}}/\wt{\ff}$ of  $(\X_2)_{\wt{\ff}}/\wt{\ff}$ realizes $\V{\wt{\Omega}}{\Omega}$, and such that the degree $d$ polynomial $\ell(x)$ has Galois group $S_{d}$ over $\wt{\ff}$. Arguing now in a manner essentially identical to the proof of \Cref{thm:main} (using \Cref{thm:zarhin-simple} \ref{item:zarhin:b} for the desired simplicity statement), we conclude that $(\X_2)_{\wt{\ff}}/\wt{\ff}$ can be specialized to give infinitely many genus $g$ hyperelliptic curves $X/K$ witnessing parts \ref{thm:quad:realize} and \ref{thm:quad:simple} of the theorem. 

Finally, if $n\leqslant 8$, then we take $g=1$ and conclude as we did in the proof of \Cref{thm:quad}.\qed
\end{emp}

\begin{emp}\label{proof-all reps}
\textit{Proof of \Cref{thm:all-reps}.} Given a Galois module $V$, we choose a \fetk $\Omega$ such that $V \subset \V{\Omega}{K}$ and apply \Cref{thm:main}.\qed
\end{emp}

\begin{rem}\label{rem:shioda simple}
In the case $\Omega = K^{4g+6}$, \Cref{thm:main} gives infinitely many genus $g$ hyperelliptic curves  $X/K$ with $J_X/K$ absolutely simple and having Mordell--Weil rank at least $4g+5$ over $K$. The case $K=\Q$ is a result of Shioda and Terasoma~\cite{terasoma}-- they showed using a different method (not relying on~\cite{zarhinSimple}) that the curve ${(\X_1)}_{\kk}$ admits infinitely many specializations $X/\Q$ such that $\End_{\qbar}((J_X)_{\qbar})=\Z$ and the specialization map $\J_1(\kk) \too J_{X}(\Q)$ is injective.
\end{rem}

\bibliographystyle{alpha}
\bibliography{real-prob}

\begin{bibdiv}
\begin{biblist}

\bib{belorousski}{book}{
      author={Belorousski, Pavel},
       title={Chow rings of moduli spaces of pointed elliptic curves},
   publisher={ProQuest LLC, Ann Arbor, MI},
        date={1998},
        ISBN={978-0591-85283-7},
         url={http://gateway.proquest.com/openurl?url_ver=Z39.88-2004&rft_val_fmt=info:ofi/fmt:kev:mtx:dissertation&res_dat=xri:pqdiss&rft_dat=xri:pqdiss:9832113},
        note={Thesis (Ph.D.)--The University of Chicago},
      review={\MR{2716762}},
}

\bib{bini-fontanari}{article}{
      author={Bini, Gilberto},
      author={Fontanari, Claudio},
       title={Moduli of curves and spin structures via algebraic geometry},
        date={2006},
        ISSN={0002-9947,1088-6850},
     journal={Trans. Amer. Math. Soc.},
      volume={358},
      number={7},
       pages={3207\ndash 3217},
         url={https://doi.org/10.1090/S0002-9947-06-03838-4},
      review={\MR{2216264}},
}

\bib{bisatt-explicit-root}{article}{
      author={Bisatt, Matthew},
       title={Explicit root numbers of abelian varieties},
        date={2019},
        ISSN={0002-9947},
     journal={Trans. Amer. Math. Soc.},
      volume={372},
      number={11},
       pages={7889\ndash 7920},
         url={https://doi.org/10.1090/tran/7926},
      review={\MR{4029685}},
}

\bib{neronmodels}{book}{
      author={Bosch, Siegfried},
      author={L\"{u}tkebohmert, Werner},
      author={Raynaud, Michel},
       title={N\'{e}ron models},
      series={Ergebnisse der Mathematik und ihrer Grenzgebiete (3) [Results in Mathematics and Related Areas (3)]},
   publisher={Springer-Verlag, Berlin},
        date={1990},
      volume={21},
        ISBN={3-540-50587-3},
         url={https://doi.org/10.1007/978-3-642-51438-8},
      review={\MR{1045822}},
}

\bib{CTrankjump}{article}{
      author={Colliot~Th{\'e}l{\`e}ne, Jean-Louis},
       title={Point g{\'e}n{\'e}rique et saut du rang du groupe de {M}ordell--{W}eil},
        date={2020},
        ISSN={1730-6264},
     journal={Acta Arithmetica},
         url={http://dx.doi.org/10.4064/aa190814-18-3},
}

\bib{deligne-mumford}{article}{
      author={Deligne, P.},
      author={Mumford, D.},
       title={The irreducibility of the space of curves of given genus},
        date={1969},
        ISSN={0073-8301,1618-1913},
     journal={Inst. Hautes \'{E}tudes Sci. Publ. Math.},
      number={36},
       pages={75\ndash 109},
         url={http://www.numdam.org/item?id=PMIHES_1969__36__75_0},
      review={\MR{262240}},
}

\bib{dokchitser-root-nonabelian}{article}{
      author={Dokchitser, Vladimir},
       title={Root numbers of non-abelian twists of elliptic curves},
        date={2005},
        ISSN={0024-6115},
     journal={Proc. London Math. Soc. (3)},
      volume={91},
      number={2},
       pages={300\ndash 324},
         url={https://doi.org/10.1112/S0024611505015261},
        note={With an appendix by Tom Fisher},
      review={\MR{2167089}},
}

\bib{dokchister-cubic}{article}{
      author={Dokchitser, Tim},
       title={Ranks of elliptic curves in cubic extensions},
        date={2007},
        ISSN={0065-1036},
     journal={Acta Arith.},
      volume={126},
      number={4},
       pages={357\ndash 360},
         url={https://doi.org/10.4064/aa126-4-5},
      review={\MR{2289966}},
}

\bib{versality}{article}{
      author={Duncan, Alexander},
      author={Reichstein, Zinovy},
       title={Versality of algebraic group actions and rational points on twisted varieties},
        date={2015},
        ISSN={1056-3911,1534-7486},
     journal={J. Algebraic Geom.},
      volume={24},
      number={3},
       pages={499\ndash 530},
         url={https://doi.org/10.1090/S1056-3911-2015-00644-0},
        note={With an appendix containing a letter from J.-P. Serre},
      review={\MR{3344763}},
}

\bib{elkiesthreelectures}{misc}{
      author={Elkies, Noam~D.},
       title={Three lectures on elliptic surfaces and curves of high rank},
        date={2007},
        note={Preprint. \href{https://arxiv.org/abs/0709.2908}{arXiv: 0709.2908}},
}

\bib{vanishing-cubic}{article}{
      author={Fearnley, Jack},
      author={Kisilevsky, Hershy},
      author={Kuwata, Masato},
       title={Vanishing and non-vanishing {D}irichlet twists of {$L$}-functions of elliptic curves},
        date={2012},
        ISSN={0024-6107},
     journal={J. Lond. Math. Soc. (2)},
      volume={86},
      number={2},
       pages={539\ndash 557},
         url={https://doi.org/10.1112/jlms/jds018},
      review={\MR{2980924}},
}

\bib{fantechi-massarenti}{article}{
      author={Fantechi, Barbara},
      author={Massarenti, Alex},
       title={On the rigidity of moduli of curves in arbitrary characteristic},
        date={2017},
        ISSN={1073-7928,1687-0247},
     journal={Int. Math. Res. Not. IMRN},
      number={8},
       pages={2431\ndash 2463},
         url={https://doi.org/10.1093/imrn/rnw105},
      review={\MR{3658203}},
}

\bib{movinglemma}{article}{
      author={Gabber, Ofer},
      author={Liu, Qing},
      author={Lorenzini, Dino},
       title={Hypersurfaces in projective schemes and a moving lemma},
        date={2015},
        ISSN={0012-7094},
     journal={Duke Math. J.},
      volume={164},
      number={7},
       pages={1187\ndash 1270},
         url={https://doi.org/10.1215/00127094-2877293},
      review={\MR{3347315}},
}

\bib{gille-szamuely}{book}{
      author={Gille, Philippe},
      author={Szamuely, Tam\'{a}s},
       title={Central simple algebras and {G}alois cohomology},
      series={Cambridge Studies in Advanced Mathematics},
   publisher={Cambridge University Press, Cambridge},
        date={2017},
      volume={165},
        ISBN={978-1-316-60988-0; 978-1-107-15637-1},
        note={Second edition of [ MR2266528]},
      review={\MR{3727161}},
}

\bib{howe-growth}{article}{
      author={Howe, Lawrence},
       title={Twisted {H}asse-{W}eil {$L$}-functions and the rank of {M}ordell-{W}eil groups},
        date={1997},
        ISSN={0008-414X},
     journal={Canad. J. Math.},
      volume={49},
      number={4},
       pages={749\ndash 771},
         url={https://doi.org/10.4153/CJM-1997-037-7},
      review={\MR{1471055}},
}

\bib{kihara14}{article}{
      author={Kihara, Shoichi},
       title={On an elliptic curve over {$\mathbb{Q}(t)$} of rank {$\geq 14$}},
        date={2001},
        ISSN={0386-2194},
     journal={Proc. Japan Acad. Ser. A Math. Sci.},
      volume={77},
      number={4},
       pages={50\ndash 51},
         url={http://projecteuclid.org/euclid.pja/1148393079},
      review={\MR{1829378}},
}

\bib{kozuma-cubic}{article}{
      author={Kozuma, Rintaro},
       title={On the rank of elliptic curves in elementary cubic extensions},
        date={2015},
        ISSN={2356-7511},
     journal={J. Numbers},
       pages={Art. ID 501629, 4},
         url={https://doi.org/10.1155/2015/501629},
      review={\MR{3413075}},
}

\bib{newpts}{article}{
      author={Liu, Qing},
      author={Lorenzini, Dino},
       title={New points on curves},
        date={2018},
        ISSN={0065-1036},
     journal={Acta Arith.},
      volume={186},
      number={2},
       pages={101\ndash 141},
         url={https://doi.org/10.4064/aa170322-23-8},
      review={\MR{3870709}},
}

\bib{lemke-thorne}{article}{
      author={Lemke~Oliver, Robert~J.},
      author={Thorne, Frank},
       title={Rank growth of elliptic curves in non-{A}belian extensions},
        date={2021},
        ISSN={1073-7928},
     journal={Int. Math. Res. Not. IMRN},
      number={24},
       pages={18411\ndash 18441},
         url={https://doi.org/10.1093/imrn/rnz307},
      review={\MR{4365991}},
}

\bib{masser-endo}{article}{
      author={Masser, D.~W.},
       title={Specializations of endomorphism rings of abelian varieties},
        date={1996},
        ISSN={0037-9484},
     journal={Bull. Soc. Math. France},
      volume={124},
      number={3},
       pages={457\ndash 476},
         url={http://www.numdam.org/item?id=BSMF_1996__124_3_457_0},
      review={\MR{1415735}},
}

\bib{matsuno-zp}{article}{
      author={Matsuno, Kazuo},
       title={A note on the growth of {M}ordell-{W}eil ranks of elliptic curves in cyclotomic {${\bf Z}_p$}-extensions},
        date={2003},
        ISSN={0386-2194},
     journal={Proc. Japan Acad. Ser. A Math. Sci.},
      volume={79},
      number={5},
       pages={101\ndash 104},
         url={http://projecteuclid.org/euclid.pja/1116443678},
      review={\MR{1980609}},
}

\bib{matsunoz2}{article}{
      author={Matsuno, Kazuo},
       title={Mordell-{W}eil ranks of elliptic curves in the cyclotomic {$\mathbb{Z}_2$}-extension of the rationals},
        date={2017},
        ISSN={1793-0421},
     journal={Int. J. Number Theory},
      volume={13},
      number={2},
       pages={429\ndash 438},
         url={https://doi.org/10.1142/S1793042117500257},
      review={\MR{3606630}},
}

\bib{mestre11}{article}{
      author={Mestre, Jean-Fran\c{c}ois},
       title={Courbes elliptiques de rang {$\geq 11$} sur $\mathbb{Q}(t)$},
        date={1991},
        ISSN={0764-4442},
     journal={C. R. Acad. Sci. Paris S\'{e}r. I Math.},
      volume={313},
      number={3},
       pages={139\ndash 142},
      review={\MR{1121576}},
}

\bib{mestre12}{article}{
      author={Mestre, Jean-Fran\c{c}ois},
       title={Courbes elliptiques de rang {$\geq 12$} sur $\mathbb{Q}(t)$},
        date={1991},
        ISSN={0764-4442},
     journal={C. R. Acad. Sci. Paris S\'{e}r. I Math.},
      volume={313},
      number={4},
       pages={171\ndash 174},
      review={\MR{1122689}},
}

\bib{neronhilbert}{article}{
      author={N\'eron, Andr\'e},
       title={Probl\`emes arithm\'etique et g\'eom\'etriques rattach\'es \`a la notion de rang d'une courbe alg\'ebrique dans un corps},
        date={1952},
     journal={Bulletin de la Soci\'et\'e Math\'ematique de France},
      volume={80},
       pages={101\ndash 166},
         url={http://www.numdam.org/item/BSMF_1952__80__101_0},
      review={\MR{0056951}},
}

\bib{neron-N1}{inproceedings}{
      author={N\'{e}ron, A.},
       title={Propri\'{e}t\'{e}s arithm\'{e}tiques de certaines familles de courbes alg\'{e}briques},
        date={1956},
   booktitle={Proceedings of the {I}nternational {C}ongress of {M}athematicians, 1954, {A}msterdam, vol. {III}},
   publisher={Erven P. Noordhoff N. V., Groningen},
       pages={481\ndash 488},
      review={\MR{87210}},
}

\bib{nagao13}{article}{
      author={Nagao, Koh-ichi},
       title={An example of elliptic curve over {$\mathbb {Q}(T)$} with rank {$\geq 13$}},
        date={1994},
        ISSN={0386-2194},
     journal={Proc. Japan Acad. Ser. A Math. Sci.},
      volume={70},
      number={5},
       pages={152\ndash 153},
         url={http://projecteuclid.org/euclid.pja/1195511052},
      review={\MR{1291171}},
}

\bib{noot}{incollection}{
      author={Noot, Rutger},
       title={Abelian varieties---{G}alois representation and properties of ordinary reduction},
        date={1995},
      volume={97},
       pages={161\ndash 171},
         url={http://www.numdam.org/item?id=CM_1995__97_1-2_161_0},
        note={Special issue in honour of Frans Oort},
      review={\MR{1355123}},
}

\bib{odoni}{article}{
      author={Odoni, R. W.~K.},
       title={The {G}alois theory of iterates and composites of polynomials},
        date={1985},
        ISSN={0024-6115},
     journal={Proc. London Math. Soc. (3)},
      volume={51},
      number={3},
       pages={385\ndash 414},
         url={https://doi.org/10.1112/plms/s3-51.3.385},
      review={\MR{805714}},
}

\bib{Qpoints}{book}{
      author={Poonen, Bjorn},
       title={Rational points on varieties},
      series={Graduate Studies in Mathematics},
   publisher={American Mathematical Society, Providence, RI},
        date={2017},
      volume={186},
        ISBN={978-1-4704-3773-2},
      review={\MR{3729254}},
}

\bib{rohrlich-family}{article}{
      author={Rohrlich, David~E.},
       title={{G}alois representations in {M}ordell--{W}eil groups of elliptic curves},
        date={2001},
     journal={Cuba Matem\'atica Educacional},
      volume={3},
      number={1},
}

\bib{rohrlich-vanishing}{incollection}{
      author={Rohrlich, David~E.},
       title={The vanishing of certain {R}ankin-{S}elberg convolutions},
        date={1990},
   booktitle={Automorphic forms and analytic number theory ({M}ontreal, {PQ}, 1989)},
   publisher={Univ. Montr\'{e}al, Montreal, QC},
       pages={123\ndash 133},
      review={\MR{1111015}},
}

\bib{rohrlich-allreps}{article}{
      author={Rohrlich, David~E.},
       title={Galois theory, elliptic curves, and root numbers},
        date={1996},
        ISSN={0010-437X},
     journal={Compositio Math.},
      volume={100},
      number={3},
       pages={311\ndash 349},
         url={http://www.numdam.org/item?id=CM_1996__100_3_311_0},
      review={\MR{1387669}},
}

\bib{rohrlich}{article}{
      author={Rohrlich, David~E.},
       title={Realization of some {G}alois representations of low degree in {M}ordell-{W}eil groups},
        date={1997},
        ISSN={1073-2780},
     journal={Math. Res. Lett.},
      volume={4},
      number={1},
       pages={123\ndash 130},
         url={https://doi.org/10.4310/MRL.1997.v4.n1.a11},
      review={\MR{1432815}},
}

\bib{sabitova-twisted-root}{article}{
      author={Sabitova, Maria},
       title={Twisted root numbers and ranks of abelian varieties},
        date={2013},
        ISSN={1942-5600},
     journal={J. Comb. Number Theory},
      volume={5},
      number={1},
       pages={25\ndash 30},
      review={\MR{3113619}},
}

\bib{serregaloiscoh}{book}{
      author={Serre, Jean-Pierre},
       title={Galois cohomology},
     edition={English},
      series={Springer Monographs in Mathematics},
   publisher={Springer-Verlag, Berlin},
        date={2002},
        ISBN={3-540-42192-0},
        note={Translated from the French by Patrick Ion and revised by the author},
      review={\MR{1867431}},
}

\bib{serremw}{book}{
      author={Serre, Jean-Pierre},
       title={Lectures on the {M}ordell-{W}eil theorem},
     edition={Third Edition},
      series={Aspects of Mathematics},
   publisher={Friedr. Vieweg \& Sohn, Braunschweig},
        date={1997},
        ISBN={3-528-28968-6},
         url={https://doi.org/10.1007/978-3-663-10632-6},
        note={Translated from the French and edited by Martin Brown from notes by Michel Waldschmidt, With a foreword by Brown and Serre},
      review={\MR{1757192}},
}

\bib{shiodaMW}{article}{
      author={Shioda, Tetsuji},
       title={Mordell-{W}eil lattices and {G}alois representation. {II}, {III}},
        date={1989},
        ISSN={0386-2194},
     journal={Proc. Japan Acad. Ser. A Math. Sci.},
      volume={65},
      number={8},
       pages={296\ndash 303},
         url={http://projecteuclid.org/euclid.pja/1195512733},
      review={\MR{1030204}},
}

\bib{shioda-elliptic-11}{article}{
      author={Shioda, Tetsuji},
       title={An infinite family of elliptic curves over {${\bf Q}$} with large rank via {N}\'{e}ron's method},
        date={1991},
        ISSN={0020-9910,1432-1297},
     journal={Invent. Math.},
      volume={106},
      number={1},
       pages={109\ndash 119},
         url={https://doi.org/10.1007/BF01243907},
      review={\MR{1123376}},
}

\bib{shioda-family}{inproceedings}{
      author={Shioda, Tetsuji},
       title={Theory of {M}ordell-{W}eil lattices},
        date={1991},
   booktitle={Proceedings of the {I}nternational {C}ongress of {M}athematicians, {V}ol. {I}, {II} ({K}yoto, 1990)},
   publisher={Math. Soc. Japan, Tokyo},
       pages={473\ndash 489},
      review={\MR{1159235}},
}

\bib{shiodasymmetry}{article}{
      author={Shioda, Tetsuji},
       title={Constructing curves with high rank via symmetry},
        date={1998},
        ISSN={0002-9327},
     journal={Amer. J. Math.},
      volume={120},
      number={3},
       pages={551\ndash 566},
         url={http://muse.jhu.edu/journals/american_journal_of_mathematics/v120/120.3shioda.pdf},
      review={\MR{1623420}},
}

\bib{silverman-specialization}{article}{
      author={Silverman, Joseph~H.},
       title={Heights and the specialization map for families of abelian varieties},
        date={1983},
        ISSN={0075-4102},
     journal={J. Reine Angew. Math.},
      volume={342},
       pages={197\ndash 211},
         url={https://doi.org/10.1515/crll.1983.342.197},
      review={\MR{703488}},
}

\bib{terasoma}{article}{
      author={Shioda, Tetsuji},
      author={Terasoma, Tomohide},
       title={Existence of simple {J}acobian varieties of genus {$g$} with rank at least {$4g+5$}},
        date={1999},
        ISSN={0002-9327},
     journal={Amer. J. Math.},
      volume={121},
      number={1},
       pages={65\ndash 72},
         url={http://muse.jhu.edu/journals/american_journal_of_mathematics/v121/121.1shioda.pdf},
      review={\MR{1704998}},
}

\bib{zarhinSimple}{article}{
      author={Zarhin, Yuri~G.},
       title={Families of absolutely simple hyperelliptic {J}acobians},
        date={2010},
        ISSN={0024-6115},
     journal={Proc. Lond. Math. Soc. (3)},
      volume={100},
      number={1},
       pages={24\ndash 54},
         url={https://doi.org/10.1112/plms/pdp020},
      review={\MR{2578467}},
}

\end{biblist}
\end{bibdiv}
\end{document}